\documentclass[14pt]{amsart}

\usepackage{bm}
 \newcommand{\tv}[1]{\bm {\mathcal{#1}}}

 \usepackage{upgreek}

\usepackage{stmaryrd} 
\usepackage[T1]{fontenc}
\usepackage{textcomp}

\usepackage{amsopn, amsthm, amsgen, amscd,amsmath,amssymb}
\usepackage[bbgreekl]{mathbbol}
\usepackage[mathcal]{euscript}
\DeclareMathAlphabet{\mathbbm}{U}{bbm}{m}{n}
\usepackage{color}
\usepackage{tikz}
  \usepackage{graphicx}
\usepackage{epstopdf}
\usepackage[small,nohug,heads = LaTeX]{diagrams} \diagramstyle[labelstyle=\scriptstyle]

\definecolor{CadetBlue}{cmyk}{0.62, 0.57, 0.23, 0 }
\definecolor{black}{cmyk}{1, 0.5, 0, 0 }
\definecolor{RedViolet}{cmyk}{0.07, 0.9, 0, 0.34 }
\definecolor{SeaGreen}{cmyk}{0.69, 0, 0.5, 0}

\DeclareMathAlphabet{\mathpzc}{OT1}{pzc}{m}{it}

\newcommand{\R}{\mathbb R}
\newcommand{\C}{\mathbb C}

\newcommand{\K}{\mathbb K}
\newcommand{\N}{\mathbb N}
\newcommand{\PR}{\mathbb P}
\newcommand{\Q}{\mathbb Q}
\newcommand{\Z}{\mathbb Z}

\newtheorem{theo}{Theorem}
\newtheorem{lemm}{Lemma}
\newtheorem{prop}{Proposition}
\newtheorem{coro}{Corollary}

\theoremstyle{definition}

\theoremstyle{remark}

\newtheorem{note}{Note}

\newcommand{\bast}{{}^{\ast}}

\newcommand{\bstar}{{}^{\star}}

\title[Arithmetic of Diophantine Approximation Groups II]{The Arithmetic of Diophantine Approximation Groups II: Mahler Arithmetic}
\author{T.M. Gendron}
\address{Instituto de Matem\'{a}ticas -- Unidad Cuernavaca, Universidad
Nacional Autonoma de M\'{e}xico, Av. Universidad S/N, C.P. 62210
Cuernavaca, Morelos, M\'{E}XICO}
\email{tim@matcuer.unam.mx}
\date{28 June 2014}
\begin{document}
\vspace{2cm}
\begin{abstract}  This is the second paper in a series of two in which a
global algebraic number theory of the reals
 is formulated with the purpose of providing a unified setting for 
algebraic and transcendental number theory.  In this paper,  to any real number $\uptheta$ we associate its {\it polynomial diophantine approximation ring} 
$ \bast \tilde{\tv{\mathcal{Z}}}(\uptheta )=\{  \bast\tilde{\tv{Z}}^{\upmu}_{\upnu}(\uptheta )_{\bast d}\} $:
a tri-filtered subring of a nonstandard model of the polynomial ring $\Z[X]$.  We characterize the filtration structure of $ \bast \tilde{\tv{\mathcal{Z}}}(\uptheta )$ according 
to the Mahler class and the Mahler type of $\uptheta$.
The {\it arithmetic} of polynomial diophantine approximation groups is introduced in terms of the {\it resultant} or {\it tensor product} of polynomials.  In particular, it is shown that 
polynomial diophantine approximation groups have the structure of  {\it approximate ideals}: wherein a partial resultant product of two polynomial diophantine approximation groups may be performed by restriction
to substructures of tri-filtration.  The explicit characterization of this partial product law is the main theorem of this paper.
 \end{abstract}
 \maketitle
\section*{Introduction}

This is the second paper in a series of two (independent) papers devoted to {\it global algebraic number theory}: an algebraic treatment of Diophantine Approximation leading to a synthesis of
algebraic and transcendental number theory into a single theory.  The accommodation of transcendental number theory into algebraic number theory is made possible 
by replacing the classical players in algebraic number theory -- Dedekind domains and their polynomial rings --
by nonstandard models of the same occurring as ultrapowers.  In our first paper \cite{Ge0}, the linear theory of Diophantine Approximation was dealt with; in this paper, we treat the nonlinear or Mahler theory, an extension
of the linear theory to approximation of real numbers by polynomials. 

As algebraic number theory issues from a study of the arithmetic of ideals in Dedekind domain,  
we incorporate transcendental number theory by introducing a generalized notion of ideal called a {\it diophantine approximation group}.
Diophantine approximation groups occur as subgroups of nonstandard models of classical Dedekind domains or polynomial rings over such.
 In particular, to  $\uptheta\in \R$ we may associate various Diophantine approximation groups 
depending on how one approximates $\uptheta$ -- by rational integers, by algebraic integers, by polynomials.  
Diophantine approximation groups come with natural filtrations -- called {\it approximate ideal structures} -- along which one can partially define products: the study of which gives rise to an arithmetic extending 
the usual arithmetic of ideals.   
In \cite{Ge0} we studied diophantine approximation groups occurring in models of Dedekind domains,  paying particular attention to how their arithmetic reflects the linear classification of real numbers.

The present paper has been written so that it can be read independently of \cite{Ge0}, and in particular, we assume no results from the latter in this work.  Nevertheless, it would be a disservice to the reader
if we did not offer first a summary of the intuition and constructions of \cite{Ge0} which are to be extended to the nonlinear setting.  We provide this summary now.

\vspace{3mm}

\begin{center}
$\diamond$
\end{center}

\vspace{3mm}

\subsection*{1. Ultrapowers.}

Given $\mathfrak{u}$ a nonprincipal ultrafilter on $\N$, the {\bf {\em ultrapower}} 
\[  \bast\Z:=\Z^{\N}/\mathfrak{u}\] is a quotient ring of the $\N$-power, whose elements are thus equivalence classes of sequences.   The corresponding ultrapower of the reals \[ \bast\R := \R^{\N}/\mathfrak{u}\supset \R\] is a field
extending $\R$.  The subgroup $ \bast\R_{\upvarepsilon}\subset\bast\R$ of classes containing a representative converging to $0$ is the group of infinitesimals and for $\bast r,\bast x\in\R$ we write 
\[ \bast x\simeq \bast y \]
if $\bast x-\bast y\in \bast\R_{\upvarepsilon}$.  The ring of bounded elements $\bast\R_{\rm fin}$ is local with maximal ideal $ \bast\R_{\upvarepsilon}$ and the residue class field is $\R$.  See \S \ref{NSStructures} for a short introduction to this construction.

\subsection*{2. Diophantine Approximation Groups.}

For $\uptheta\in\R$, the  {\bf {\em diophantine approximation group}}  \cite{Ge1}, \cite{Ge0} \[ \bast\Z (\uptheta)\subset \bast\Z\]
is the subgroup of $\bast n\in\bast\Z$ for which there exists $\bast n^{\perp}\in\bast\Z$ such that
\[ \bast n\uptheta -\bast n^{\perp} \simeq 0.\]  
The dual element $\bast n^{\perp}$ 
is uniquely determined by $\bast n$ and we refer to $(\bast n^{\perp},\bast n)$ as a ``numerator denominator pair'',  denoting it here using a suggestive pseudo fractional notation
\[ 
 \begin{array}{l}
 \bast n^{\perp} \\
 \widetilde{\bast n\;\;}
 \end{array}. 
       \]
       When $\uptheta = a/b\in\Q$ then $\bast\Z (\uptheta)=\bast(b)$ = the ultrapower of the ideal $(b)$. Otherwise $\bast\Z (\uptheta)$ is only a group and $\bast\Z (\uptheta )\cap\Z=0$:
  that is, to ``observe'' $\uptheta$ by way of $\Z$ it is essential that we leave the standard model.   $\bast\Z (\uptheta)$
  provides a notion of fundamental group for the Kronecker foliation of slope $\uptheta$ and so also plays a central role in the definition of the quantum modular invariant \cite{CaGe}.

One would like to manipulate numerator denominator pairs using the usual fractional arithmetic e.g. if  $\uptheta,\upeta\in \R$ and
\[ \begin{array}{l}
 \bast m^{\perp} \\
 \widetilde{\bast m\;\;}
 \end{array},\quad
  \begin{array}{l}
 \bast n^{\perp} \\
 \widetilde{\bast n\;\;} 
  \end{array}\]
  are numerator denominator pairs associated to
$\bast m\in \bast\Z (\uptheta)$,
$\bast n\in \bast\Z (\upeta)$, we would like to assert that
\[    \begin{array}{l}
 \bast m^{\perp} \\
 \widetilde{\bast m\;\;}
 \end{array} 
 \cdot 
  \begin{array}{l}
 \bast n^{\perp} \\
 \widetilde{\bast n\;\;} 
  \end{array}:=  \begin{array}{l}
\bast m^{\perp}\cdot \bast n^{\perp} \\
 \widetilde{\bast m\cdot\bast n\;\;\;\;}
 \end{array},\quad 
  \begin{array}{l}
 \bast m^{\perp} \\
 \widetilde{\bast m\;\;}
 \end{array} 
 \pm
  \begin{array}{l}
 \bast n^{\perp} \\
 \widetilde{\bast n\;\;} 
  \end{array}:=
   \begin{array}{l}
\left(\bast m\bast n^{\perp}\pm\bast m^{\perp}\bast n\right) \\
 \widetilde{{}\quad\quad\bast m\cdot\bast n\;\;\;\;\quad{}}
 \end{array}\]be numerator denominator pairs associated to
 diophantine approximations of $ \uptheta \upeta, \uptheta\pm\upeta$.
  It is not difficult to see that this cannot be true unconditionally except when $\uptheta,\upeta\in\Q$. The description and study of conditions under which such a fractional arithmetic holds 
  is in a sense the central theme in this series of papers.
  
\subsection*{3. Approximate Ideal Structure.}\label{Para3}

A setting under which fractional arithmetic of numerator denominator pairs holds may be described by specifying a bi-filtration of subgroups 
  \[ \bast \Z (\uptheta ) = \{\bast\Z^{\upmu}_{\upnu}(\uptheta )\}.\]
 The indices $\upmu, \upnu$ are elements of the quotient 
  \[ \bstar\PR\R = \bast \R_{+}/(\bast\R_{\rm fin})_{+}^{\times} \]
  which has the structure of a totally ordered tropical (max-times) ring, associated to the valuation
  \[ \bast\R \longrightarrow \bstar\PR\R, \quad \bast r\longmapsto \langle \bast r\rangle := |\bast r| \cdot (\bast\R_{\rm fin})_{+}^{\times}. \]
  Then
  $\bast\Z^{\upmu}_{\upnu}(\uptheta )$ is the subgroup of elements $\bast n\in \bast \Z (\uptheta ) $ for which the {\bf {\em growth}} $\bast n^{-1}$ satisfies
   \[ \langle \bast n^{-1}\rangle >\upmu\] and whose {\bf {\em decay}} 
  \[ \upvarepsilon (\bast n) := \bast n\uptheta -\bast n^{\perp}\]
  satisfies 
  \[   \langle  \upvarepsilon (\bast n)\rangle \leq \upnu.\]  
  Then in \cite{Ge0} we proved that the ordinary product induces a bilinear map
  \begin{align}\label{gdprodintro} \bast\Z^{\upmu}_{\upnu}(\uptheta )\times\bast\Z^{\upnu}_{\upmu}(\upeta )\stackrel{\cdot}{\longrightarrow} \bast\Z^{\upmu\cdot\upnu}(\uptheta \upeta)\cap\bast\Z^{\upmu\cdot\upnu}(\uptheta +\upeta) \cap\bast\Z^{\upmu\cdot\upnu}(\uptheta -\upeta).
\end{align}
As a consequence, whenever $\bast m\in  \bast\Z^{\upmu}_{\upnu}(\uptheta )$ and $\bast n\in \bast\Z^{\upnu}_{\upmu}(\upeta )$, then their numerator denominator pairs may be multiplied and 
added/subtracted exactly as formulated in Paragraph 2 above.  
When $\uptheta =a/b$, $\upeta=c/d\in\Q$, (\ref{gdprodintro}) reduces to the product map
 $\bast (b)\times \bast (d)\rightarrow \bast (bd)$ of the principal ideals generated by the denominators.  The pairing (\ref{gdprodintro}) will be a subcase
 of a more general pairing proved in this paper for polynomial diophantine approximations, described further below. 
 
 The concept of an approximate ideal generalizes naturally that of ideal as follows.  If we consider just the ``growth filtration''
$\bast\Z = \{\bast \Z^{\upnu}\}$ where $\bast\Z^{\upnu}=\{\bast n|\; \upnu<\upmu(\bast n)\}$ then for each $\upmu,\upnu\in \bstar\PR\R_{\upvarepsilon}$, 
\[ \bast  \Z^{\upnu}\cdot \bast\Z^{\upmu}_{\upnu}(\uptheta )\subset \bast\Z^{\upmu\cdot\upnu}(\uptheta ).\] 
 By forgetting the indices one recovers  the usual definition of an ideal.  

The group ${\rm PGL}_{2}(\Z )$ -- which defines equivalence of real numbers -- induces isomorphisms of diophantine approximation groups preserving
their approximate ideal structures e.g. if $A\in {\rm PGL}_{2}(\Z )$ then for all $\upmu, \upnu$, $A$ defines an isomorphism
\[ \bast\Z^{\upmu}_{\upnu}(\uptheta )\cong  \bast\Z^{\upmu}_{\upnu}(A(\uptheta) ). \]

\subsection*{4. Nonvanishing Spectra.}

We define the {\bf {\em nonvanishing spectrum}} of $\uptheta$ to be
\[ {\rm Spec}(\uptheta)=\{ (\upmu,\upnu)|\; \bast\Z^{\upmu}_{\upnu}(\uptheta )\not=0\}.\]
In \cite{Ge0}, we characterized the linear classification of the reals -- rational, badly approximable, (very) well approximable and Liouville -- in terms of their
nonvanishing spectra, see Figure 1 below.  In this paper, we will obtain spectral portraits for polynomial diophantine approximations which are organized according
to the Mahler classification.

\begin{figure}[htbp]\label{Spectralportraits}
\centering
\includegraphics[width=5in]{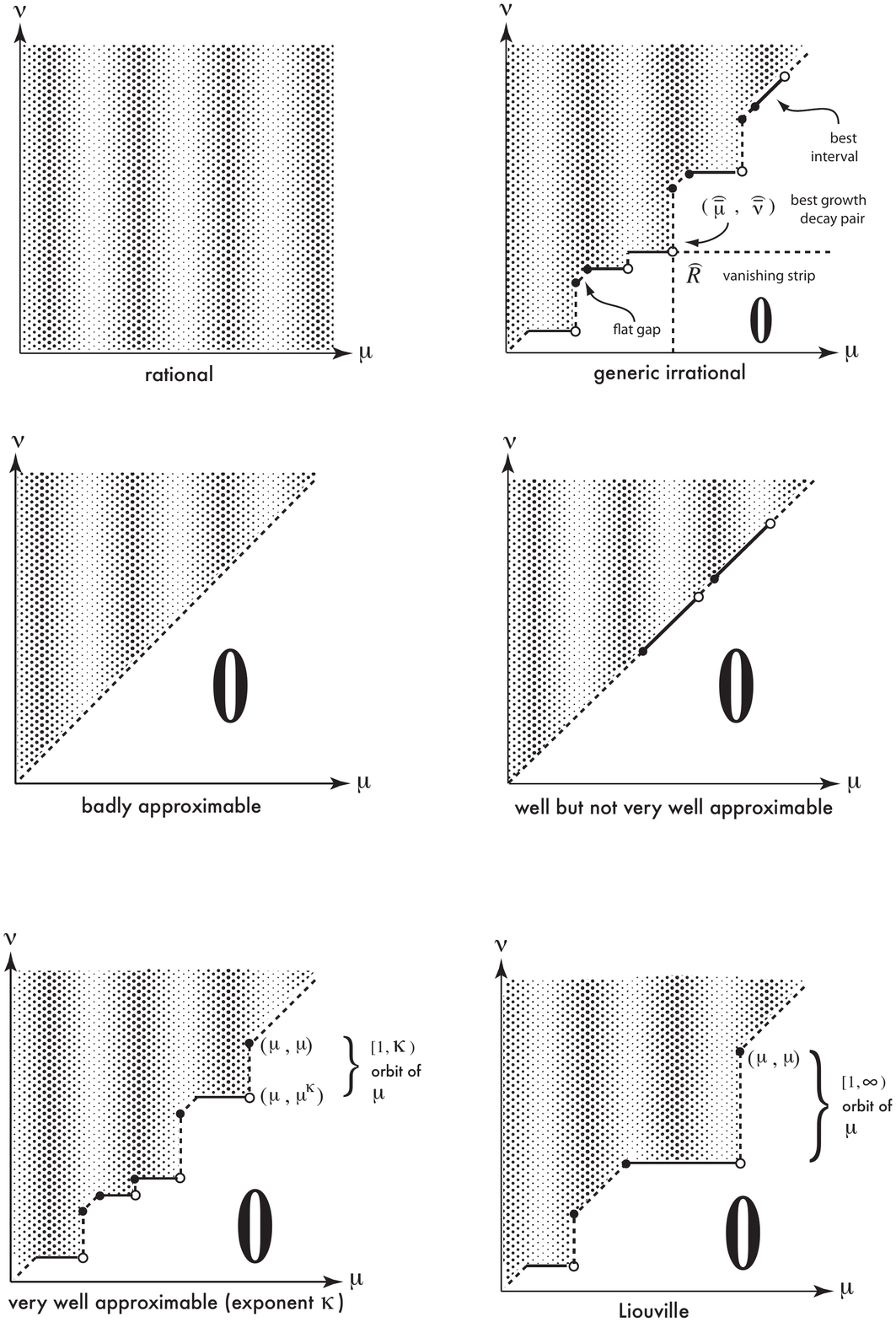}
\caption{Portraits of Spectra.  Shaded regions and heavy lines represent nonvanishing.}\label{portraits}
\end{figure}

\subsection*{5. Composability}  For $\upmu\geq\upnu$ we define the  ``composability'' relation 
\[\uptheta {}_{\upmu}\!\!\owedge_{\upnu}\upeta\]
whenever the groups appearing in the product (\ref{gdprodintro}) are nontrivial i.e.\ for $(\upmu,\upnu)\in {\rm Spec}(\uptheta)$, $(\upnu,\upmu)\in {\rm Spec}(\upeta)$.
  Roughly speaking, composability increases as one progresses
from the badly approximable numbers to the Liouville numbers.  In this connection a new phenomenon emerges: the existence of {\bf {\em  antiprimes}} -- classes of numbers for which the relation ${}_{\upmu}\!\owedge_{\upnu}$ is empty
for all  possible growth-decay parameters.  The largest anti prime set is the set $\mathfrak{B}$ of badly approximable numbers.  

\subsection*{6. $K$-Diophantine Approximation Groups}
Let
$K/\Q$ be a finite extension, $\mathcal{O}$ the ring of $K$-integers and $\K\cong \R^{d}$ the Minkowski space of $K$. The diophantine approximation group of $\boldsymbol z\in\K$ has the structure of an
approximate ideal 
\[ \bast\mathcal{O}(\boldsymbol z)=\{ \bast \mathcal{O}^{\boldsymbol\upmu}_{\boldsymbol\upnu}(\boldsymbol z)\}. \] These $K$-approximate ideals
may be multiplied according to an obvious 
analogue of (\ref{gdprodintro}).  If $K/\Q$ is Galois, then the action of ${\rm Gal}(K/\Q)$ on $\K$ extends to an action
on growth-decay indices so that the growth-decay product becomes Galois natural.  The $K$-nonvanishing spectrum ${\rm Spec}_{K}(\boldsymbol z)$ may be used to define the nontrivial classes of
$K$-badly approximable, $K$-(very) well approximable and $K$-Liouville elements of $\K$.
One observes the phenomenon of {\bf {\em antiprime splitting}}, where a $\Q$-badly approximable number $\uptheta$ loses
its antiprime status upon diagonal inclusion in $\K$: this happens for quadratic Pisot-Vijayaraghavan numbers. 

\subsection*{7. Approximate Ideal Classes}
The {\bf {\em approximate ideal class}} of $\bast\mathcal{O}(\boldsymbol z)$ 
is defined
\[ \bast[\mathcal{O}](\boldsymbol z ) :=  \bast \mathcal{O}(\boldsymbol z)+\bast \mathcal{O}(\boldsymbol z)^{\perp},\quad  \bast \mathcal{O}(\boldsymbol z)^{\perp} := \{ \bast \upalpha^{\perp}|\;
\bast\upalpha\in   \bast \mathcal{O}(\boldsymbol z)\}.\]
The set of  approximate ideal classes $\mathcal{C}l (\K)$ extends the usual ideal class group $\mathcal{C}l (K)$ of $K/\Q$: if 
\[ \mathfrak{a} = (\upalpha,\upbeta ),\quad \mathfrak{a}' = (\upalpha',\upbeta' )\subset\mathcal{O}\]  are classical ideals and $\upgamma = \upalpha/\upbeta,\upgamma' = \upalpha'/\upbeta'$ then 
\[ \bast [\mathcal{O}](\upgamma)=\bast [\mathcal{O}](\upgamma') \Longleftrightarrow [\mathfrak{a}]=[\mathfrak{a}'] \;\; \text{(equality of ideal classes)}.\]
There is a canonical surjective map  \[ {\rm PGL}_{2}(\mathcal{O})\backslash\K \longrightarrow \mathcal{C}l (\K)\]
which extends the bijection ${\rm PGL}_{2}(\mathcal{O})\backslash K\leftrightarrow   \mathcal{C}l (K)$ and which is conjecturally a bijection as well.
 When $K=\Q$,  ${\rm PGL}_{2}(\Z)\backslash\R$ is the moduli space of quantum tori.
 
 \vspace{3mm}

\begin{center}
$\diamond$
\end{center}

\vspace{3mm}

We now turn to the contents of the present paper: the arithmetic of polynomial diophantine approximation.  
Denote the ultraproduct of $\Z[X]$ \[ \bast\tilde{\tv{Z}}=\bast (\Z [X]) = \Z[X]^{\N}/\mathfrak{u},\] the ring of possibly infinite degree nonstandard polynomials. 
Then we may define the {\it ring} of polynomial diophantine approximations
\[ \bast\tilde{\tv{Z}}(\uptheta )= \{ \bast f(X)\in \bast\tilde{\tv{Z}}|\; \bast f(\uptheta )\simeq 0\}  \]
which acquires the structure of approximate ideal indexed not just by growth and decay but also by a third index which sets a bound $\bast d$ on the degree e.g.
\[ \bast\tilde{\tv{Z}}(\uptheta )=\{  \bast\tilde{\tv{Z}}^{\upmu}_{\upnu}(\uptheta )_{\bast d}\} .\]
In \S \ref{polyspectramahlersection} we study the 
nonvanishing spectrum 
and characterize
the Mahler classes and Mahler types according to the portraits of their nonvanishing spectra e.g. see Figure \ref{Mahlerportraits}, Theorems \ref{Atheo}--\ref{Utheo} and Theorems \ref{SType}--\ref{UType}.

The definition of approximate ideal arithmetic in the polynomial setting is dictated by the wholly reasonable demand that the inclusion
\[  \bast\Z (\uptheta )\hookrightarrow\bast\tilde{\tv{Z}}(\uptheta ),\quad \bast n\mapsto \bast n X-\bast n^{\perp},\]
be an approximate ideal monomorphism.  The nonlinear analogue of fractional arithmetic of linear equations is provided by the {\bf {\em resultant sum}} 
and the {\bf {\em resultant product}}
\[ f\boxplus g,\quad f\boxtimes  g\] of polynomials $f,g$, which are essentially characterized by the property that their root sets are the sum and the product, respectively, of the root sets of $f$ and $g$.  The resultant product
was first defined in \cite{BC}, \cite{Gl} (where it is referred to as the tensor product); in \S \ref{resultantarithmetic} we develop the basic properties that we require of them from scratch.

Both
$\boxplus$ and $\boxtimes$ distribute over the (Cauchy) product; if we let $\check{\tv{Z}}$
be the Cauchy monoid of integral non-0 polynomials modulo multiplication by non-0 integers, then $\boxplus$ and $\boxplus$
define operations on $\check{\tv{Z}}$ making the latter a double monoid, in which the inclusion $\Q\hookrightarrow \check{\tv{Z}}$,
$a/b\mapsto aX-b$,
takes the operations $+$ to $\boxplus$ and $\times$ to $\boxtimes$.  Thus we can view $\check{\tv{Z}}$ as a generalized
field extension of $\Q$; since $\boxtimes$ does not distribute over $\boxplus$, it is a spiritual relative of the nonlinear number field construction of \cite{GeVe}.

Approximate ideal arithmetic of the resultant product $\boxtimes$ for finite degree polynomials is studied in \S \ref{polygdasec};
the infinite degree case presents particular complications and is deferred to
 \S\ref{polygdasec2}.  The most general version of the approximate ideal product is given by Theorem \ref{semivirtualarith} of \S\ref{polygdasec2}, which asserts the existence of a product
\begin{align*}
\bast\tilde{\tv{Z}}^{\upmu}_{(\upmu^{1-\bast e^{-1}}\cdot\upnu)^{\bast e}}(\uptheta )_{\bast d} \times\bast\tilde{\tv{Z}}^{\upnu}_{(\upmu\cdot\upnu^{1-\bast d^{-1}})^{\bast d}}(\upeta )_{\bast e}\stackrel{\boxtimes}{\longrightarrow}  \bast\tilde{\tv{Z}}^{\upmu\cdot\upnu}(\uptheta\upeta )_{\bast d\bast e}
\end{align*}  
that is Cauchy bilinear and agrees with (\ref{gdprodintro}) on the images of $\bast\Z(\uptheta)\subset \bast\tv{Z}(\uptheta )$,$\bast\Z (\upeta)\subset\bast\tv{Z}(\upeta )$. 
An analogue of the composability relation ${}_{\upmu}\!\!\owedge_{\upnu}$
is introduced and its satisfaction according to the
Mahler classes of its arguments is studied: composability increasing as one passes through the  $S$-number, $T$-number, $U$-number hierarchy, the $S$-numbers providing
essentially an antiprime set, see  Theorem \ref{Anumarith}.

\vspace{3mm}

\noindent {\bf  Acknowledgements.}  I would like to thank Gregor Weingart for suggesting the proof of Theorem \ref{resulininteral} and Jos\'{e} Antonio de la Pe\~{n}a for pointing out the reference \cite{Gl}.   This paper was supported in part by the CONACyT grant 058537 as well as the PAPIIT grant IN103708.

\section{Ultrapowers}\label{NSStructures}

This brief section contains all the reader will need to know about ultrapowers, see also \cite{ChKei}, \cite{Go}.

Let $I$ be a set.  A {\bf filter} on $I$ is a subset $\mathfrak{f}\subset \text{\sf 2}^{I}$ satisfying
\begin{enumerate}
\item[-] If $X, Y\in \mathfrak{f}$ then $X\cap Y\in \mathfrak{f}$.
\item[-] If $X\in \mathfrak{f}$ and $X\subset Y$ then $Y\in \mathfrak{f}$.
\item[-] $\emptyset\not\in \mathfrak{f}$.
\end{enumerate}
Any set $\mathcal{F}\subset \text{\sf 2}^{I}$ satisfying the finite intersection property generates a filter, denoted $\langle \mathcal{F}\rangle$.
A maximal filter $\mathfrak{u}$ is called an {\bf ultrafilter}.  Equivalently, a filter $\mathfrak{u}$ is an ultrafilter $\Leftrightarrow$ for all $X\in \text{\sf 2}^{I}$, $X\in \mathfrak{u}$ or $I-X\in \mathfrak{u}$.
An ultrafilter $\mathfrak{u}$ is {\bf principal} if it contains a finite set $F$: equivalently $\mathfrak{u}=\langle F\rangle$.  Otherwise it is {\bf nonprincipal}.  By Zorn's lemma, every filter is contained in an ultrafilter.

Now let $\{G_{i}\}_{i\in I}$ be a family of algebraic structures of a fixed type: for our purposes, they will be groups, rings, fields. Let $\mathfrak{u}$ be an ultrafilter on $I$. The quotient
\[  \prod_{i\in I} G_{i}/\sim_{\mathfrak{u}},\quad (g_{i})\sim_{\mathfrak{u}} (g_{i}') \Longleftrightarrow \{ i|\; g_{i}=g_{i}'\}  \in\mathfrak{u} \]
is called the {\bf ultraproduct} of the $G_{i}$ w.r.t.\ $\mathfrak{u}$.  By the Fundamental Theorem of Ultraproducts (\L o\'{s}'s Theorem) \cite{ChKei}, the ultraproduct is also a group/ring/field according to the case.  If $G_{i}=G$ for all $i$ the ultraproduct is called an {\bf ultrapower} and is denoted 
\[ \bast G=\bast G_{\mathfrak{u}}.\]
Elements of $\bast G$ will be denoted 
\[ \bast g = \bast \{ g_{i}\}.\]
The canonical inclusion $G\hookrightarrow\bast G$ given by constants $g\mapsto(g=g_{i})$ is a monomorphism.  If $\mathfrak{u}$ is nonprincipal, this map is not onto and 
again by \L o\'{s}, exhibits $\bast G$ as a {\bf nonstandard model} of $G$: that is, the set of sentences in first order logic satisfied by $\bast G$ coincides with that of $G$.

If $I=\N$ and $\mathfrak{u}$ is a nonprincipal ultrafilter on $\N$ we denote by \[ \bast\Z\subset\bast\Q\subset\bast\R\subset\bast\C\] corresponding ultrapowers of $\Z\subset \Q\subset \R\subset \C$.  The field $\bast \R$ is totally ordered and the absolute value $|\cdot |$ extends to a map $|\cdot|:\bast \R\rightarrow \bast \R_{+}\cup \{0\}$.
We define the local subring of bounded elements
\[ \bast\R_{\rm fin} := \{\bast r\in\bast\R|\; \exists M\in\R_{+}\text { such that } |\bast r|< M\} \]
whose maximal ideal is the ideal of {\bf infinitesimals}
\[ \bast\R_{\upvarepsilon}  :=\{ \bast r\in\bast\R_{\rm fin}|\; \forall M\in \R_{+},\; |\bast r|< M \}.\]
Then $\bast\R$ is the field of fractions of $\bast\R_{\rm fin}$ and the residue class field is 
\[ \bast\R_{\rm fin}/ \bast\R_{\upvarepsilon}\cong \R.\]

\section{Tropical Growth-Decay Semi-Ring}\label{tropical}

Let $\bast\Z\subset \bast\R$ be ultrapowers of $\Z\subset\R$ with respect to a fixed non-principal ultrafilter on $\N$.  As usual 
$\bast \R_{\rm fin}\subset \bast\R$ denotes the subring of bounded sequence classes and $\bast\R_{\upvarepsilon}\subset \bast \R_{\rm fin}$
the maximal ideal of infinitesimals. The subscript ``$+$'' will refer to
the positive elements of any of the above sets e.g.\ $\bast\Z_{+}\subset\bast\Z$ is the subset
of positive elements in $\bast\Z$.  

 Let
$(\bast\R_{\rm fin})_{+}^{\times}$ = the group of positive units in the ring $\bast\R_{\rm fin}$; thus $(\bast\R_{\rm fin})_{+}^{\times}$ is
the multiplicative
subgroup of noninfinitesimal, noninfinite elements in $\bast\R_{+}$.
Consider the multiplicative quotient group 
\[  \bstar\PR\R := 
\bast\R_{+}/(\bast\R_{\rm fin})_{+}^{\times},\] 
whose elements will be written 
\[ \upmu = \bast x\cdot(\bast\R_{\rm fin})_{+}^{\times}.\]  We denote the product
in $\bstar\PR\R$ by ``$\cdot$''.

\begin{prop}\label{integerrepre} Every element $\upmu\in \bstar\PR\R$ may be written in the form 
\[ \bast n^{\upvarepsilon}\cdot(\bast\R_{\rm fin})_{+}^{\times}\]
where $\bast n\in\bast \Z_{+}-\Z_{+}$ or $\bast n =1$, and $\upvarepsilon=\pm 1$.
\end{prop}

 \begin{proof} Every element of $\bstar\PR\R$ is the class of 1, the class of an infinite element or the class of an infinitesimal element. 
If $\upmu$ is the class of $\bast r$ infinite, then there exists $\bast \bar{r}\in [0,1)=\{\bast x|\; 0\leq \bast x <1\}$ and $\bast n\in\bast\Z_{+}$
for which $\bast r= \bast n+\bast  \bar{r} = \bast n \cdot ((\bast n+\bast  \bar{r})/\bast n)$.  But  
$(\bast n+\bast  \bar{r})/\bast n = 1+\bast\bar{r}/\bast n\in(\bast\R_{\rm fin})_{+}^{\times}$, so  $\upmu = \bast n\cdot (\bast\R_{\rm fin})_{+}^{\times}$.  A similar argument may be used to show that when $\upmu$ represents an infinitesimal class, $\upmu =\bast n^{-1}\cdot(\bast\R_{\rm fin})_{+}^{\times}$ for some $\bast n\in\bast \Z_{+}-\Z_{+}$.
\end{proof}

\begin{prop}\label{denselinord}  $ \bstar\PR\R$ is a densely ordered group.
\end{prop}

\begin{proof}  The order is defined
by declaring that $\upmu<\upmu'$ in  $\bstar\PR\R$ if for any pair of representatives
$\bast x\in\upmu, \bast x'\in\upmu'$ we have $ \bast x<\bast x'$, evidently a dense order without endpoints. The left-multiplication action 
of $\bast \R_{+}$ on $ \bstar\PR\R$ preserves this order, therefore so does the product: if $\upmu <\upnu$ then for all $\xi\in \bstar\PR\R$, $\xi\cdot \upmu <\xi\cdot \upnu$.  
\end{proof}

We introduce the maximum of a pair of elements in $ \bstar\PR\R$ 
as a formal binary operation:
\[ \upmu+ \upnu:= \max (\upmu, \upnu).\]
The operation $+$
is clearly commutative and associative.
The following Proposition says that $+$ is the quotient of the operation $+$ of $\bast\R^{\times}_{+}$. 

\begin{prop} Let $\upmu = \bast x\cdot  (\bast\R_{\rm fin})_{+}^{\times}$, $\upmu'=\bast x'\cdot  (\bast\R_{\rm fin})_{+}^{\times}$.
Then \[ (\bast x+\bast x')\cdot (\bast\R_{\rm fin})_{+}^{\times}=
\upmu+ \upmu'.\]
\end{prop}

\begin{proof} Note that $\bast x+\bast x'\in\bast\R_{+}$ and
$ \bast x+\bast x'\in \max (\upmu, \upmu')$.   Indeed, suppose first that $\upmu\not=\upmu'$, say $\upmu<\upmu'$.  Then there exists $\bast \upvarepsilon$ infinitesimal
for which $\bast x = \bast \upvarepsilon\bast x'$, and we have   
$   \bast x' + \bast x = \bast x' (1+\bast \upvarepsilon)\in \upmu'$.       
If $\upmu=\upmu'$ then  $\bast x' = \bast r\bast x$ for $\bast r\in \bast\R_{+}^{\times}$ and 
$ (\bast x+ \bast x')\cdot  (\bast\R_{\rm fin})_{+}^{\times} = \bast x (1+\bast r)\cdot  (\bast\R_{\rm fin})_{+}^{\times} = \upmu = \upmu +\upmu$.
\end{proof}

\begin{prop}\label{troprop}  Let $\bast r,\bast s\in\bast\R_{+}$ and $\upmu,\upnu,\upnu'\in \bstar\PR\R$. Then
\begin{enumerate}
\item[1.]  $\upmu \cdot (\upnu+ \upnu') =  (\upmu \cdot \upnu )+ (\upmu \cdot \upnu' )$.
\item[2.]  $\bast r \cdot (\upnu+ \upnu')= (\bast r\cdot\upnu )+ (\bast r\cdot\upnu' )$.
\item[3.] $(\bast r+\bast s)\cdot\upmu = (\bast r\cdot\upmu)+ (\bast s\cdot\upmu)$.
\end{enumerate}
\end{prop}

\begin{proof}  1.  It is enough to check the equality in the case $\upnu'>\upnu$.  Then $\upmu \cdot (\upnu+ \upnu')=\upmu\cdot\upnu'$.
But the latter is equal to $ (\upmu \cdot \upnu )+ (\upmu \cdot \upnu' )$ since the product preserves
the order.  The proof of 2. is identical, where we use the fact that the multiplicative action by
$\bast\R_{+}$ preserves the order.  Item 3. is trivial.
\end{proof}

It will be convenient to add the class $-\infty$ of the element $0\in\bast\R$
to the space $\bstar\PR\R$: in other words, we will
reconsider $\bstar\PR\R$ as the quotient 
$ (\bast\R_{+}\cup\{0\})/(\bast\R_{\rm fin})_{+}^{\times}$. 
Note that
we have for all $\upmu\in\bstar\PR\R$
\[  -\infty+ \upmu = \upmu, \quad  -\infty\cdot \upmu =-\infty .\]
In particular, $-\infty$ is the neutral element for the operation $+$.
Thus, by Proposition \ref{troprop}:

\begin{theo} $\bstar\PR\R$
is an abstract (multiplicative)
 tropical semi-ring: that is, a max-times semi ring.
\end{theo}

We will refer to $\bstar\PR\R$ as the {\bf growth-decay semi-ring}.  Let $\bstar \PR\R_{\upvarepsilon} \subset\bstar\PR\R$ be the image of the $(\bast\R_{\rm fin})_{+}^{\times}$-invariant multiplicatively closed set $(\bast\R_{\upvarepsilon})_{+}$.  With the operations $\cdot, +$, $\bstar \PR\R_{\upvarepsilon}$ is a sub tropical semi-ring: the {\bf decay semi-ring}.

 If we forget the tropical addition, considering $ \bstar\PR\R$ as a linearly ordered multiplicative group, then
the map 
\[ \langle \cdot \rangle :\bast\R\rightarrow \bstar\PR\R, \quad \langle\bast x\rangle = |\bast x|\cdot(\bast\R_{\rm fin})_{+}^{\times},\] is the Krull valuation associated 
to the local ring $\bast\R_{\rm fin}$
(see for example \cite{ZS}).  The restriction of $\langle \cdot \rangle$ to $\R$ is just the trivial valuation, so
that $\langle \cdot \rangle $ cannot be equivalent to the usual valuation $|\cdot |$ on $\bast\R$ induced from the euclidean norm.
Note also that $\langle \cdot \rangle$
is nonarchimedean.  We refer to $\langle\cdot\rangle$ as the {\bf growth-decay valuation}.

There is a natural ``Frobenius'' action of the multiplicative group $\R^{\times}_{+}$ on $\bstar\PR\R$:
for $\upmu\in \bstar\PR\R_{\upvarepsilon}$ and $\bast x\in\upmu$  
define
\[\Upphi_{r}(\upmu )=\upmu^{r} := \bast x^{r}\cdot (\bast\R_{\rm fin})_{+}^{\times}\]
for each $r\in\R_{+}^{\times}$. 
Note that this action does not depend on the choice of representative $\bast x$.  
We may extend the Frobenius action to $(\bast\R_{\rm fin})_{+}^{\times}$ as follows.
For $\bast r=\bast\{r_{i}\}\in   (\bast\R_{\rm fin})_{+}^{\times}$ and $\upmu\in\bstar\PR\R$ represented by $\bast x =\bast \{x_{i}\}\in\bast\R_{+}$ 
define
\[ \Upphi_{\bast r}(\upmu )=\upmu^{\bast r} :=   \bast \{ x_{i}^{r_{i}} \}\cdot (\bast\R_{\rm fin})_{+}^{\times}, \]
which is again well-defined.  Note that it is {\it not} the case that if $\bast r\simeq r\in\R_{+}$ that $\upmu^{\bast r}=\upmu^{r}$.

\begin{theo}\label{FrobProp}  The map $\Upphi_{\bast r}:\bstar\PR\R\rightarrow \bstar\PR\R$ is a tropical automorphism for each $\bast r\in(\bast\R_{\rm fin})^{\times}_{+}$ 
and defines a faithful representation
\[  \Upphi:(\bast\R_{\rm fin})^{\times}_{+}\longrightarrow {\rm Aut}(\bstar\PR\R).\]
\end{theo}

\begin{proof} $\Upphi_{\bast r}$ is clearly multiplicative.  
Moreover:  $(\upmu+\upnu )^{\bast r}=(\max (\upmu, \upnu ))^{\bast r}=\upmu^{\bast r}+\upnu^{\bast r}$. 
\end{proof} 

We denote by $\bar{\upmu}$ the orbit of $\upmu$ by $(\bast\R_{\rm fin})^{\times}_{+}$ with respect to $\Upphi$.
Note that by Theorem \ref{FrobProp}:
\begin{enumerate}
\item[-] $\bar{\upmu}$ is a sub tropical semi-ring of $\bstar\PR\R$. 
\item[-] The quotient of $\bstar\PR\R$ by $\Upphi$, denoted $\bstar\overline{\PR\R}$,
is a tropical semi-ring.  
\end{enumerate}

\begin{prop}\label{FrobConHomeo}  For each $\upmu\in\bstar\PR\R$, $\upmu\not=-\infty, 1$, $\bar{\upmu}$ is connected: if $\upmu',\upmu''\in\bar{\upmu}$
and $\upmu'<\upmu''$ then $[\upmu',\upmu'']\in\bar{\upmu}$.  Moreover, the map $\bast r\mapsto \upmu^{\bast r}$ defines a homeomorphism
$(\bast\R_{\rm fin})^{\times}_{+}\approx\bar{\upmu}$ with respect to the order topologies.
\end{prop}

\begin{proof}  Given $1\not=\bast x\in\bast\R_{+}$ and 
$\bast r,\bast s\in (\bast\R_{\rm fin})^{\times}_{+}$,
it is immediate that for any $\bast y$ in the interval defined by 
$\bast x^{\bast r},\bast x^{\bast s}$,  there exists $\bast t\in (\bast r,\bast s)$ such that $\bast y=\bast x^{\bast t}$.  From this the connectedness
statement follows.  Now suppose that there exists $\bast t\in\bast\R_{+}$ and $\upmu\not=\infty, 1$ such that $\upmu^{\bast t}=1$.  Without loss of generality assume
that $\upmu = \bast n\cdot (\bast\R_{\rm fin})^{\times}_{+}$ is infinite.  Then there exists $\bast a\in (\bast\R_{\rm fin})^{\times}_{+} $ such that $\bast a = \bast n^{\bast t}$.  But such a $\bast t$ would necessarily be infinitesimal, contrary to our hypothesis. \end{proof}

We denote by $[\bast x]$ the image in $\bstar\overline{\PR\R}$ of $\bast x\in\bast\R$.
For all $ \bar{\upmu}, \bar{\upnu}\in \bstar\overline{\PR\R}$, we write
$ \bar{\upmu}<\bar{\upnu}$ $ \Leftrightarrow$ for all $\upmu\in \bar{\upmu}$, $\upnu\in \bar{\upnu}$,
$\upmu < \upnu $.

\begin{prop}\label{FrobDLO}  $\bstar\overline{\PR\R}$ is a dense linear order.
\end{prop}

\begin{proof}  If $ \bar{\upmu}\not<\bar{\upnu}$ and $ \bar{\upmu}\not>\bar{\upnu}$ then 
it follows that there exist representatives $\upmu\in \bar{\upmu}$, $\upnu\in \bar{\upnu}$
for which $\upmu<\upnu$ and $\upmu^{\bast r}>\upnu$ for $\bast r\in (\bast\R_{\rm fin})^{\times}_{+}$.
We may assume without loss of generality that both $\upmu,\upnu$ represent infinite classes
so that $\bast r>1$.  Representing $\bast x=\bast\{x_{i}\}\in\upmu$ and $\bast y=\bast\{y_{i}\}\in\upnu$, let
$\bast s=\bast\{s_{i}\}$ where $s_{i}$ is the unique positive real satisfying $x_{i}^{s_{i}}=y_{i}$.  Then
$\bast s\in [1,\bast r]\subset(\bast\R_{\rm fin})^{\times}_{+}$, $\upmu^{\bast s}=\upnu$ and therefore $\bar{\upmu}=\bar{\upnu}$.  Thus $\bstar\overline{\PR\R}$ is a linear order.
On the other hand, if $ \bar{\upmu}<\bar{\upnu}$, then choosing representatives $\bast x$, $\bast y$ as above, we have
$\bast x^{\bast s}=\bast y$ for $\bast s$ infinite.  If we let $\upmu'$ be the class of $\bast x^{\sqrt{\bast s}}$, then
$\bar{\upmu}<\bar{\upmu}'<\bar{\upnu}$.
\end{proof}

We call $\bstar\overline{\PR\R}$ the {\bf  Frobenius growth-decay semi-ring}.

\begin{note}\label{infinitesimalpower}  In fact, we may extend $\Upphi$ to an action of $(\bast\R_{\upvarepsilon})_{+}$ = the multiplicative monoid of positive infinitesimals on $\bstar\PR\R$.  That is, infinitesimal powers $\upmu^{\bast \upvarepsilon}$, $\bast\upvarepsilon\in(\bast\R_{\upvarepsilon})_{+}$,
are well-defined, since $\bast r^{\bast \upvarepsilon}\in (\bast\R_{\rm fin})_{+}^{\times}$ for all $\bast r\in (\bast\R_{\rm fin})_{+}^{\times}$.  This induces
an action of $(\bast\R_{\upvarepsilon})_{+}$ on $\bstar\overline{\PR\R}$, $\bar{\upmu}\mapsto \bar{\upmu}^{\bast \upvarepsilon}$ by tropical homomorphisms which
either expand or contract the order according to whether $\bar{\upmu}$ is the class of an infinitesimal or infinite element: e.g. $\bar{\upmu}<\bar{\upmu}^{\bast \upvarepsilon}$ if $\bar{\upmu}$ is an infinitesimal class.  We will say that $\bar{\upmu}\in \bstar\overline{\PR\R}_{\upvarepsilon}$
is a $\bast\boldsymbol\upvarepsilon$ {\bf root of unity} if $\bar{\upmu}^{\bast \upvarepsilon}=\bar{1}$.  The set of $\bast\upvarepsilon$ roots
of unity union $-\infty$ clearly forms a sub tropical semi ring of $\bstar\overline{\PR\R}_{\upvarepsilon}$.
\end{note}

Define the {\bf  tropical subtraction} of classes $\upmu,\upnu\in\bstar\PR\R$ by 
\[ \upmu-\upnu=\upnu-\upmu := \min (\upmu,\upnu).\]
Tropical subtraction satisfies $\upmu-\upnu\geq \upmu\cdot\upnu$ in $\bstar\PR\R_{\upvarepsilon}$, and descends to a well defined binary operation in
$\bstar\overline{\PR\R}$.  Let $\bstar \overline{\PR\R}_{\upvarepsilon}$ be the image of the 
decay semi-ring in $\bstar \overline{\PR\R}$.  

\begin{prop}\label{tropsubistropprod}  Tropical subtraction coincides with tropical product in $\bstar\overline{\PR\R}_{\upvarepsilon}$: $\bar{\upmu}\cdot\bar{\upnu} = \bar{\upmu}-\bar{\upnu}$
for all $\bar{\upmu},\bar{\upnu}\in \bstar\overline{\PR\R}_{\upvarepsilon}$.
\end{prop}

\begin{proof}  If $\bar{\upmu}=\bar{\upnu}$ then $\bar{\upmu}\cdot\bar{\upmu}=\bar{\upmu}=\bar{\upmu}-\bar{\upmu}$. Now assume that
$\bar{\upmu}>\bar{\upnu}$.  Then there exists $\bast \updelta\in \bast\R_{\upvarepsilon}$ such that $\bar{\upmu}=\bar{\upnu}^{\bast \updelta}$.
Then $\bar{\upmu}\cdot\bar{\upnu}=\bar{\upnu}^{1+\bast \updelta}=\bar{\upnu}=\bar{\upmu}-\bar{\upnu}$.

\end{proof}

\section{Polynomial Approximate Ideals and Rational Maps}\label{polyideologysection}

In this section we begin our study of the approximate ideal structure of polynomial diophantine
approximation.  In order to moderate notation, we will henceforward write 
\[  
\bast \tv{\mathcal{Z}}:= \bast\Z [X]\subset  \bast \tilde{\tv{\mathcal{Z}}}:= \bast(\Z [X]).  \]
Note that an element $\bast f\in \bast \tv{\mathcal{Z}}$ is a finite degree polynomial in the ring $\bast\Z $
whereas an element $\bast f=\bast\{ f_{i}\}\in \bast \tilde{\tv{\mathcal{Z}}}$ has a well-defined possibly infinite degree
$\deg (\bast f) =\bast \{ \deg (f_{i})\}\in\bast\N$.  For each $\bast d\in\bast \N$ we define the additive subgroup
\[  \bast \tilde{\tv{\mathcal{Z}}}_{\bast d}:= \{\bast f\in  \bast \tilde{\tv{\mathcal{Z}}}|\; \deg (\bast f)\leq \bast d\}. \]
This gives $ \bast \tilde{\tv{\mathcal{Z}}}$ the structure of a filtered ring:
\[ \bast \tilde{\tv{Z}}_{\bast d}\cdot \bast \tilde{\tv{Z}}_{\bast e}\subset \bast \tilde{\tv{Z}}_{\bast d+\bast e} .\]

We will also consider the following coarsening of the degree filtration.  Let
\begin{align*}
 \bast\breve{d} & =  \{\bast d'\in\bast\N |\; \langle \bast d'\rangle =\langle \bast d\rangle \} \\
 & = \{\bast d'\in\bast\N |\;\exists \bast p\in (\bast\Q_{\rm fin})^{\times}_{+} \text{ such that } \bast d'=\bast p\bast d  \}.
 \end{align*}
The set of such classes $\bast\breve{d}$ is ordered; note that
\[   \bast\breve{d}+ \bast\breve{e} = (\bast  d+\bast e)\breve{\;}. \]
Indeed, if say $ \bast\breve{d}< \bast\breve{e}$ and $\bast d'=\bast p\bast d$, $\bast e'=\bast q\bast e$ where $\bast p,\bast q\in (\bast\Q_{\rm fin})^{\times}_{+} $
then $\bast d'+\bast e'= \bast x(\bast d+\bast e)$ where
\[  \bast x = \frac{\bast p\bast d+\bast q\bast e}{\bast d+\bast e} \simeq\bast q\in(\bast\Q_{\rm fin})^{\times}_{+} ,\]
since $\bast d/\bast e$ is infinitesimal.
Then for each $\bast \breve{d}$,
\[\bast \tilde{\tv{\mathcal{Z}}}_{\bast \breve{d}}:=   \bigcup_{\bast d\in\bast \breve{d}}  \bast \tilde{\tv{\mathcal{Z}}}_{\bast d}\]
is a {\it ring} filtered by subgroups and 
\[ \bast \tilde{\tv{Z}}_{\bast \breve{d}}\cdot \bast \tilde{\tv{Z}}_{\bast \breve{e}}\subset \bast \tilde{\tv{Z}}_{\bast \breve{d}+\bast \breve{e}}\]
makes of $\bast \tilde{\tv{Z}}$ a ring filtered by rings.  Note that $\bast \tilde{\tv{Z}}_{\breve{1}} =\bast \tv{Z}$.

For each $\uptheta\in\R$,  the {\bf ring of (infinite-degree) polynomial diophantine approximations} is defined
\[  \bast \tilde{\tv{\mathcal{Z}}}(\uptheta )= \{\bast f\in \bast \tilde{\tv{\mathcal{Z}}}|\; \bast f(\uptheta )\simeq 0\} .\]
The subring $\bast \tv{\mathcal{Z}}(\uptheta )=  \bast \tilde{\tv{\mathcal{Z}}}(\uptheta )\cap \bast \tv{\mathcal{Z}}$
was studied in \S 6 of \cite{Ge4}.  Defining  $ \bast \tilde{\tv{Z}}(\uptheta)_{\bast \breve{d}}=  \bast \tilde{\tv{\mathcal{Z}}}(\uptheta )\cap \bast \tilde{\tv{Z}}_{\bast \breve{d}}$
induces on $ \bast \tilde{\tv{Z}}(\uptheta)$ a filtered ring structure as well.

We would like to measure growth and decay of polynomials in $ \bast \tilde{\tv{Z}}(\uptheta)$.  
Classically, there are a number of ways one measures the size of a polynomial (c.f. the Appendix to \S 3 of \cite{Wa1}).  We will be
interested in two: the height and the Mahler measure.  Scaling by degree makes them tropical equivalent in a sense to made precise below.

For a standard polynomial $f\in\Z[X]$ denote by $\mathfrak{h}(f)$ the {\bf  height}: the maximum of the absolute values of the coefficients
 of $f$.  Note that the height is subadditive. For $\bast f=\bast \{f_{i}\}\in\bast \tilde{\tv{Z}}$ define
 \[ \mathfrak{h}(\bast f) = \bast \{ \mathfrak{h}(f_{i})\}\in\bast \N.\]
 
 On the other hand: recall that if $f(X)=a\prod (X-\upalpha_{i})$ its {\bf  Mahler measure} is 
 \[  \mathfrak{m}(f)= |a|\prod\max(1,|\upalpha_{i}|). \]
 By definition, the Mahler measure is multiplicative.
 The Mahler measure of $\bast f\in \bast\tilde{\tv{Z}}$ is defined in the customary way: 
 \[  \mathfrak{m}(\bast f)= \bast \{ \mathfrak{m}(f_{i})\}\in\bast\R_{+}\cup \{ 0\}.\]
 The relationship between the height and the Mahler measure is given by the following double inequality (see page 113 of \cite{Wa1}):
 \begin{align}\label{heightmahlequiv}  (\deg(\bast f) +1)^{-1/2}\cdot\mathfrak{m}(\bast f)  \leq\mathfrak{h}(\bast f)\leq 2^{\deg(\bast f)}\cdot\mathfrak{m}(\bast f). \end{align}

It will be to our advantage to
scale growth according to degree.  That is, for each $\bast d\in \bast \N$ we will define a $\bast d$-scaled notion of size for 
elements of $ \bast \tilde{\tv{\mathcal{Z}}}_{\bast d}$.  In particular, the size of a fixed polynomial $\bast f$ will be, as such, a relative measure, relative
to the group $ \bast \tilde{\tv{\mathcal{Z}}}_{\bast d}$ in which it is being considered.  For $\bast f\in \bast \tilde{\tv{\mathcal{Z}}}_{\bast d}$ define
\[ \mathfrak{h}_{\bast d}(\bast f) := \mathfrak{h}(\bast f)^{1/\bast d},\quad \mathfrak{m}_{\bast d}(\bast f) := \mathfrak{m}(\bast f)^{1/\bast d}.\]
Notice that by (\ref{heightmahlequiv})
\[   \langle\mathfrak{h}_{\bast d}(\bast f)\rangle = \langle \mathfrak{m}_{\bast d}(\bast f) \rangle \in \bstar\PR\R .\]
Thus, measuring degree-normalized growth in $\bstar\PR\R$ permits us to use either the height or the Mahler measure according to the exigencies
of the situation we find ourselves in, as the paragraphs which follow illustrate.

In fact, for reasons which will soon become clear, it is natural to begin the parametrization of growth and decay in the polynomial setting using the Frobenius tropical semi ring $\bstar\overline{\PR\R}_{\upvarepsilon}$ (introduced at the end of \S \ref{tropical})
rather than $\bstar\PR\R_{\upvarepsilon}$.  For any $\bast d\geq \deg (\bast f)$ define
\[   \bar{\upmu}_{\bast d}(\bast f) := [\mathfrak{h}_{\bast d}(\bast f)^{-1}] = [ \mathfrak{m}_{\bast d}(\bast f)^{-1}]\in\bstar\overline{\PR\R}_{\upvarepsilon},\]
where we recall that $[\bast x]$ is the image in $\bstar\overline{\PR\R}$ of $\bast x\in\bast\R$.
Note that if $\bast d'\in\bast\breve{d}$ then $ \bar{\upmu}_{\bast d'}(\bast f) = \bar{\upmu}_{\bast d}(\bast f) $: thus the $\bast \breve{d}$-normalized
expression
 \[   \bar{\upmu}_{\bast \breve{d}}(\bast f)= \bar{\upmu}_{\bast d}(\bast f) \]
 is well-defined.
In particular, when $\bast d=d\in\N$ we have
$ \bar{\upmu}_{\breve{d}}(\bast f) =  \bar{\upmu}_{\breve{1}}(\bast f)=[\mathfrak{h}(\bast f)^{-1}]$.
Thus we may define
\[  \bast \tilde{\tv{\mathcal{Z}}}_{\bast \breve{d}}^{\bar{\upmu}} = \big\{ \bast f\in  \bast \tilde{\tv{\mathcal{Z}}}_{\bast \breve{d}}\big|\; 
 \bar{\upmu}_{\bast \breve{d}}(\bast f) >\bar{\upmu}\big\}= 
 \big\{ \bast f\in  \bast \tilde{\tv{\mathcal{Z}}}_{\bast \breve{d}}\big|\; 
\mathfrak{h}_{\bast d}(\bast f) \cdot \bar{\upmu}=\mathfrak{m}_{\bast d}(\bast f) \cdot \bar{\upmu}\in \bstar\overline{\PR\R}_{\upvarepsilon}\big\}
  \]
  (where in the last expression above $\bast d\in\bast\breve{d}$).
Due to the subadditivity of the height, we have, assuming, say, $ \mathfrak{h}_{\bast d}(\bast f)\geq  \mathfrak{h}_{\bast d}(\bast g)$:
\begin{align*}
  \mathfrak{h}_{\bast d}(\bast f+\bast g) \cdot \bar{\upmu} & \leq \big(\mathfrak{h}(\bast f) + \mathfrak{h}(\bast g) \big)^{1/\bast d}\cdot \bar{\upmu} \\
& \leq 2^{1/\bast d}\, \mathfrak{h}_{\bast d}(\bast f)\cdot \bar{\upmu} \in  \bstar\overline{\PR\R}_{\upvarepsilon}
\end{align*}
 so that $\bast \tilde{\tv{\mathcal{Z}}}_{\bast \breve{d}}^{\bar{\upmu}}$ is a group.
Moreover, by the multiplicativity of the Mahler measure, we also have:
\begin{align}\label{frobgradedringbyrings} 
\bar{\upmu}_{\bast \breve{d}+\bast \breve{e}}(\bast f\cdot\bast g) & = [(\mathfrak{m}(\bast f)\cdot\mathfrak{m}(\bast g))^{1/(\bast d+\bast e)}]^{-1}\\
& \geq  \bar{\upmu}_{\bast \breve{d}}(\bast f)\cdot \bar{\upmu}_{\bast \breve{e}}(\bast g)  \nonumber
\end{align}
for $\bast f\in  \bast \tilde{\tv{\mathcal{Z}}}_{\bast \breve{d}}$ and $\bast g\in  \bast \tilde{\tv{\mathcal{Z}}}_{\bast \breve{e}}$.  
In particular, $ \bast \tilde{\tv{\mathcal{Z}}}_{\bast \breve{d}}^{\bar{\upmu}}$ is a ring 
(take $\bast\breve{d}=\bast\breve{e}$ in (\ref{frobgradedringbyrings}) and note that $2\bast\breve{d}=\bast\breve{d}$, $\bar{\upmu}^{2}=\bar{\upmu}$), 
\[    \bast \tilde{\tv{\mathcal{Z}}}_{\bast \breve{d}}^{\bar{\upmu}}\cdot  \bast \tilde{\tv{\mathcal{Z}}}_{\bast \breve{e}}^{\bar{\upmu}'}\subset
 \bast \tilde{\tv{\mathcal{Z}}}_{\bast \breve{d}+\bast \breve{e}}^{\bar{\upmu}\cdot\bar{\upmu}'}
  \]
and therefore $\bast\tilde{\tv{\mathcal{Z}}}=\{ \bast\tilde{\tv{\mathcal{Z}}}_{\bast \breve{d}}^{\bar{\upmu}}\}$ has the structure of a ring doubly filtered by rings.
  
  Now given $\uptheta\in\R$ and $\bast f\in \bast \tilde{\tv{\mathcal{Z}}}(\uptheta )_{\bast \breve{d}}$ we define the normalized decay via
  \[\bar{\upnu}_{\bast \breve{d}}(\bast f) := [|\bast f(\uptheta )|^{1/\bast d}],\quad \bast d\in\bast \breve{d}\]
and the subring
\[  \bast \tilde{\tv{Z}}_{\bar{\upnu}}(\uptheta)_{\bast \breve{d}} =\{ \bast f\in \bast \tilde{\tv{Z}}(\uptheta )_{\bast \breve{d}}|\; 
\bar{\upnu}_{\bast \breve{d}} (\bast f)\leq \bar{\upnu}\}. \]
This filtration by rings has the expected product law:
\[\bast \tilde{\tv{Z}}_{\bar{\upnu}}(\uptheta)_{\bast \breve{d}}\cdot \bast \tilde{\tv{Z}}_{\bar{\upnu}'}(\uptheta)_{\bast \breve{e}}\subset 
\bast \tilde{\tv{Z}}_{\bar{\upnu}\cdot\bar{\upnu}'}(\uptheta)_{\bast \breve{d}+\bast \breve{e}}.
\]
Write as before 
\[ \bast \tilde{\tv{Z}}^{\bar{\upmu}}_{\bar{\upnu}}(\uptheta)_{\bast \breve{d}} = \bast \tilde{\tv{Z}}_{\bar{\upnu}}(\uptheta)_{\bast \breve{d}}
\cap  \bast \tilde{\tv{\mathcal{Z}}}_{\bast \breve{d}}^{\bar{\upmu}}.\]
Finally, note that for all $\bast \breve{d}$ and $n\in\N$, the power map $\bast f\mapsto \bast f^{n}$ induces a multiplicative inclusion  
\begin{align}\label{powerlawmap} \bast \tilde{\tv{Z}}^{\bar{\upmu}}_{\bar{\upnu}}(\uptheta)_{\bast \breve{d}}\subset \bast \tilde{\tv{Z}}^{\bar{\upmu}}_{\bar{\upnu}}(\uptheta)_{\breve{d}}.
\end{align}

\begin{theo}\label{polyideology}  The triply indexed filtration \[ \bast \tilde{\tv{Z}}(\uptheta) =\{\bast \tilde{\tv{Z}}^{\bar{\upmu}}_{\bar{\upnu}}(\uptheta)_{\bast \breve{d}}\}\] by subrings
makes of $\bast \tilde{\tv{Z}}(\uptheta)$ a ring triply filtered by rings.  Moreover,
\begin{align}\label{polyideologydef}\bast \tilde{\tv{Z}}^{\bar{\upnu}}_{\bast \breve{d}} \cdot \bast \tilde{\tv{Z}}^{\bar{\upmu}}_{\bar{\upnu}}(\uptheta)_{\bast \breve{d}} \subset\bast \tilde{\tv{Z}}^{\bar{\upmu}\cdot\bar{\upnu}}(\uptheta)_{\bast \breve{d}}.
\end{align} 
\end{theo}

\begin{proof}  The proof of the first statement is contained in the paragraphs preceding the Theorem.  As for the second statement, 
note 
that for all $\bast g\in \bast \tilde{\tv{Z}}_{\bast \breve{d}}$ with $|\bast g(\uptheta)|\geq 1$, $\bast d\in\breve{d}$,
\begin{align*} [|\bast g(\uptheta)|^{1/\bast d}] 
& \leq [ \uptheta^{\deg (\bast g)/\bast d}(\deg (\bast g)\cdot \mathfrak{h}(\bast g))^{1/\bast d} ]\\
& = [\mathfrak{h}_{\bast d}(\bast g)]  .
\end{align*}
If $|\bast g(\uptheta)|\leq 1$ then it is trivial that $ [|\bast g(\uptheta)|^{1/\bast d}]\leq \bar{1}\leq  [\mathfrak{h}_{\bast d}(\bast g)]$.
Then for any $\bast f\in \bast \tilde{\tv{Z}}^{\bar{\upmu}}_{\bar{\upnu}}(\uptheta)_{\bast \breve{d}}$ and 
$\bast g\in  \bast \tilde{\tv{Z}}^{\bar{\upnu}}_{\bast \breve{d}}$ 
\begin{align*} 
\bar{\upnu}_{\bast \breve{d}}(\bast g\cdot\bast f) & = [| \bast g(\uptheta)\cdot\bast f(\uptheta)|^{1/\bast \breve{d}}] \\
&\leq  [\mathfrak{h}_{\bast \breve{d}}(\bast g)]\cdot \bar{\upnu} \in \bstar\overline{\PR\R}_{\upvarepsilon} .
 \end{align*}
\end{proof}

We may thus speak of $\bast \tilde{\tv{Z}}(\uptheta)$ as being a {\bf  polynomial approximate ideal} (the latter concept defined by equation
(\ref{polyideologydef}) of Theorem \ref{polyideology}).  Since it is indexed by the Frobenius tropical semi-ring, we refer to it as the {\bf  Frobenius polynomial approximate ideal}.
We may view $\bast \tilde{\tv{Z}}^{\bar{\upmu}}_{\bar{\upnu}}(\uptheta)_{\bast \breve{d}}$ as an approximate ideal within the ring $\bast \tilde{\tv{Z}}_{\bast \breve{d}}$.  In particular, the Frobenius approximate ideal can be regarded as providing a filtration of $\bast \tilde{\tv{Z}}(\uptheta)$ by approximate ideals,
indexed by the degree class $\bast\breve{d}$.

In order to bring the type of a Mahler class into play, we introduce as well
the {\bf ordinary polynomial approximate ideal} 
\[  \bast \tilde{\tv{Z}}(\uptheta ) = \{ \bast \tilde{\tv{Z}}^{\upmu}_{\upnu}(\uptheta )_{\bast d}\} \]
 which is defined exactly as the Frobenius approximate ideal, using $\bstar\PR\R_{\upvarepsilon}$ in place of $\bstar\overline{\PR\R}_{\upvarepsilon}$ to index 
 degree normalized growth and decay.  In particular, the ordinary approximate ideal makes of $\bast \tilde{\tv{Z}}(\uptheta )$ a ring tri-filtered by groups, 
 which satisfies the following
 analogue of Theorem \ref{polyideology}:
 
 \begin{theo}  For all $\bast d\leq \bast e\in\bast \N$ and $\upmu,\upnu\in\bstar\PR\R_{\upvarepsilon}$,
\[\bast \tilde{\tv{Z}}^{\upnu}_{\bast d} \cdot \bast \tilde{\tv{Z}}^{\upmu}_{\upnu}(\uptheta)_{\bast e} \subset\bast \tilde{\tv{Z}}^{\upmu\cdot\upnu}(\uptheta)_{\bast d+\bast e}.
 \]
 \end{theo}
 
 \begin{proof}  Fix $\bast g\in \bast \tilde{\tv{Z}}^{\upnu}_{\bast d}$ and $\bast f\in \bast \tilde{\tv{Z}}^{\upmu}_{\upnu}(\uptheta)_{\bast e}$. By the same argument employed in the proof of Theorem \ref{polyideology}, $\langle |\bast g(\uptheta)|^{1/\bast d}\rangle\leq  \langle\mathfrak{h}_{\bast d}(\bast g)\rangle$. Moreover, 
\begin{align*} 
\upnu_{\bast d+\bast e}(\bast g\cdot\bast f) & = \langle | \bast g(\uptheta)\cdot\bast f(\uptheta)|^{1/\bast d+\bast e}\rangle \\
&\leq  \mathfrak{h}_{\bast d}(\bast g)^{\bast d/(\bast d+\bast e)} \cdot \upnu^{\bast e/(\bast d+\bast e) } \\
&\leq (\mathfrak{h}_{\bast d}(\bast g)\cdot \upnu)^{\bast e/(\bast d+\bast e) }
\in \bstar\PR\R_{\upvarepsilon} .
 \end{align*}
 \end{proof}
 
The ordinary approximate ideal gives a refinement of the Frobenius approximate ideal:
\[  \bast \tilde{\tv{Z}}^{\bar{\upmu},\upmu}_{\bar{\upnu},\upnu}(\uptheta )_{\bast \breve{d},\bast d}:=\bast \tilde{\tv{Z}}^{\upmu}_{\upnu}(\uptheta )_{\bast d}\cap  \bast \tilde{\tv{Z}}^{\bar{\upmu}}_{\bar{\upnu}}(\uptheta )_{\bast \breve{d}}  \]
whenever $\upmu\in \bar{\upmu}$, $\upnu\in \bar{\upnu}$ and $\bast d\in\bast \breve{d}$.  Since $\bar{\upmu}^{2}=\bar{\upmu}$, $\bar{\upnu}^{2}=\bar{\upnu}$
and $2\bast \breve{d}=\bast \breve{d}$,
\[ \bast \tilde{\tv{Z}}^{\bar{\upmu},\upmu}_{\bar{\upnu},\upnu}(\uptheta )_{\bast \breve{d},\bast d}\cdot  \bast \tilde{\tv{Z}}^{\bar{\upmu},\upmu'}_{\bar{\upnu},\upnu'}(\uptheta )_{\bast \breve{d},\bast d'}\subset  \bast \tilde{\tv{Z}}^{\bar{\upmu},\upmu\cdot\upmu'}_{\bar{\upnu},\upnu\cdot\upnu'}(\uptheta )_{\bast \breve{d},\bast d+\bast d'}.\] 
Therefore, the ordinary approximate ideal endows each ring $\bast \tilde{\tv{Z}}^{\bar{\upmu}}_{\bar{\upnu}}(\uptheta )_{\bast \breve{d}}$ with the structure
of a ring tri-filtered by groups.  

A ring homomorphism
\[ \Phi :  \bast \tilde{\tv{Z}}(\uptheta )\longrightarrow \bast \tilde{\tv{Z}}(\upeta )\]
preserving the Frobenius approximate ideal structure (and the ordinary approximate ideal structure) will be called a(n) {\bf  (ordinary) Frobenius approximation homomorphism}, and if $\Phi$
is invertible, a(n) {\bf  (ordinary) Frobenius approximation isomorphism}.  This latter notion of equivalence is too strong for our purposes
so we weaken it in two essential ways.

\begin{enumerate}
\item[1.] First, we will need to allow that $\Phi$ be merely multiplicative (not necessarily preserving the additive structure), that is, we
will consider the approximate ideals $ \bast \tilde{\tv{Z}}(\uptheta )$ as simply semigroups, referring to them as {\bf semigroup approximate ideals}, and use the terminology of the previous paragraph to qualify semigroup homomorphisms.  It will be clear in this section and those that follow that in the polynomial theory, the multiplicative structure
plays the role of the passive operation (``external sum'');  the role of active operations (``multiplication, internal sum/difference'') will be played by the resultant product, sum and difference see \S \ref{resultantarithmetic}.
\item[2.]  A pair $\Phi= (\Phi_{1},\Phi_{2})$ of injective  Frobenius semigroup homomorphisms 
\[ \Phi_{1} :  \bast \tilde{\tv{Z}}(\uptheta )\longrightarrow \bast \tilde{\tv{Z}}(\upeta ), \quad 
\Phi_{2} :  \bast \tilde{\tv{Z}}(\upeta )\longrightarrow \bast \tilde{\tv{Z}}(\uptheta )\]
for which there exists $m, n\in\N$ such that
\[ \Phi_{1} (\bast \tilde{\tv{Z}}^{\upmu^{m}}_{\upnu^{m}}(\uptheta )_{\bast d})\subset \bast \tilde{\tv{Z}}^{\upmu}_{\upnu}(\upeta )_{\bast d\cdot m},
\quad
\Phi_{2} (\bast \tilde{\tv{Z}}^{\upmu^{n}}_{\upnu^{n}}(\upeta )_{\bast d})\subset \bast \tilde{\tv{Z}}^{\upmu}_{\upnu}(\upeta )_{\bast d\cdot n}
\]
will be called an {\bf  approximation consensus}, and we denote the equivalence relation
defined by an approximation  consensus by 
\[ \bast \tilde{\tv{Z}}(\uptheta )\rightleftharpoons \bast \tilde{\tv{Z}}(\upeta ).\]
Note that on the level of the Frobenius structure, an approximation  consensus gives a pair of injective semigroup homomorphisms respecting
the Frobenius tri-filtration: 
\[ \Phi_{1} (\bast \tilde{\tv{Z}}^{\bar{\upmu}}_{\bar{\upnu}}(\uptheta )_{\bast \breve{d}})\subset \bast \tilde{\tv{Z}}^{\bar{\upmu}}_{\bar{\upnu}}(\upeta )_{\bast \breve{d}},
\quad
\Phi_{2} (\bast \tilde{\tv{Z}}^{\bar{\upmu}}_{\bar{\upnu}}(\upeta )_{\bast \breve{d}})\subset \bast \tilde{\tv{Z}}^{\bar{\upmu}}_{\bar{\upnu}}(\upeta )_{\bast \breve{d}}.
\]
\end{enumerate}

We now turn our attention to the search for a notion of equivalence of real numbers which is adapted to the dictates of polynomial diophantine
approximation.  
We will insist that such a notion of equivalence implies an approximation consensus between corresponding approximate ideals. 
Since the equivalence appropriate to linear diophantine approximation is projective linear (see Paragraph 3 of the Introduction), it stands to reason  
that higher degree rational maps will provide a relevant notion of equivalence in the polynomial milieu.

Thus: let $\text{\sf Rat}_{\Z}$ be the set of rational maps that may be written in the ``relatively prime form'' 
\[ R(X)=\frac{p(X)}{q(X)}\] where $p(X),q(X)\in\Z [X]$ are relatively prime and $q(X)$ is not the zero polynomial.
We equip $\text{\sf Rat}_{\Z}$ with the operation $\circ$ of composition; its field operations will play no role in what follows.

Let us describe the relatively prime form of the composition $R\circ S$ of $R,S\in\text{\sf Rat}_{\Z}$.
First, recall that the degree of $R$ is defined $\deg (R)=\max (\deg p, \deg q)$, and that \cite{Bea}
\[ \deg (R\circ S) = \deg(R)\deg(S).\]  
If we write $p(X)=\sum_{i=0}^{\deg p}p_{i}X^{i}$, $q(X)=\sum_{j=0}^{\deg q}q_{j}X^{j}$ and $S(X)=t(X)/u(X)$
with $t(X),u(X)$ relatively prime, then clearing denominators gives
\begin{align}\label{compofratmap}   (R\circ S)(X)= u(X)^{\deg q-\deg p}\cdot\frac{\sum_{i=0}^{\deg p}p_{i}t(X)^{i}u(X)^{\deg p-i}}{\sum_{j=0}^{\deg q}q_{j}t(X)^{j}u(X)^{\deg q-j}}.  \end{align}
Suppose that $\deg p\geq \deg q$ and $\deg t\geq \deg u$.  Then the degree of the numerator of (\ref{compofratmap}) is $\deg p\deg t$ and the
degree of the denominator is 
\[ \deg q\deg t + (\deg p-\deg q)\deg u\leq \deg p\deg t.\]  This implies that the numerator and denominator of (\ref{compofratmap}) are relatively prime: for otherwise, we would obtain 
$\deg (R\circ S)<\deg(R)\deg(S)$.  The other three cases ($\deg p\leq \deg q$ and $\deg t\geq \deg u$, etc.) produce the same result, namely,
that (\ref{compofratmap}) is the relatively prime form of $R\circ S$.

Write 
\[ \uptheta  \gtrdot \upeta\] if there exists $R\in\text{\sf Rat}_{\Z}$ such that $R(\uptheta )=\upeta$.  If $\upeta  \gtrdot \uptheta$ as well,
we will write \[\uptheta \doteq \upeta\] and say that $\uptheta ,\upeta\in\R$ are $\text{\sf Rat}_{\Z}$-{\bf equivalent}.
Note that $\uptheta\doteq\upeta$ $\Leftrightarrow$ $\Q (\uptheta )= \Q (\upeta )$, and that
$ \uptheta\doteq\upeta$ $\Rightarrow$ $\uptheta$ and $\upeta$ are algebraically dependent.  The converse is 
false e.g.\ if $\Q (\uptheta,\upeta)$ is not purely transcendental.


Given $R(X)=p(X)/q(X)\in\text{\sf Rat}_{\Z}$ and $f(X)=\sum_{i=0}^{d} a_{i}X^{i}\in \Z [X]$, define
\begin{align}\label{squareoper}  ( f \boxcircle R)(X) := \sum_{i=0}^{d}a_{i}p(X)^{i}q(X)^{n-i} = q(X)^{d}\cdot (f\circ R)(X). \end{align} 
This definition extends in the usual way to any $\bast f\in\bast\tilde{\tv{Z}}$.  
Note that
\[ \deg (\bast f \boxcircle R) = \deg (\bast f)\cdot \deg (R). \]
 In particular, 
 \[  \bast\tilde{\tv{Z}}_{\bast d} \boxcircle R\subset  \bast\tilde{\tv{Z}}_{\bast d\cdot \deg(R)}\quad\text{and}\quad 
  \bast\tilde{\tv{Z}}_{\bast \breve{d}} \boxcircle R\subset  \bast\tilde{\tv{Z}}_{\bast \breve{d}}\]


\begin{prop} The map 
\[ (\bast\tilde{\tv{Z}},\cdot )\times\text{\sf Rat}_{\Z}\longrightarrow (\bast\tilde{\tv{Z}},\cdot )\]
described by (\ref{squareoper}) defines an action of $\text{\sf Rat}_{\Z}$ by multiplicative monomorphisms.  That is, 
for all $\bast f,\bast g\in\bast\tilde{\tv{Z}}$ and $R,S\in \text{\sf Rat}_{\Z}$,
\begin{enumerate}
\item  $(\bast f\cdot \bast g) \boxcircle R =   (\bast f \boxcircle R)\cdot (\bast g \boxcircle R)$.
\item $( \bast f \boxcircle R) \boxcircle S= 
  \bast f \boxcircle (R\circ S)$.
  \end{enumerate}
  \end{prop}
  
  \begin{proof}  Item (1) is clear.  It is enough to prove (2) for a standard polynomial $f$ of degree $d$.  Suppose that $R=p(X)/q(X)$, $S(X)=t(X)/u(X)$
  are parametrized as before, and that $\deg R=\deg p$.  Then
  \begin{align*}  
  ( f \boxcircle R) \boxcircle S & = \big(q(X)^{d}\cdot (f\circ R)(X)\big) \boxcircle S \\
   & = q(t(X)/u(X))^{d}\cdot (f\circ (R\circ S))(X)\cdot u(X)^{d\deg p} \\
   & = \big(   u(X)^{\deg p-\deg q}\cdot\sum q_{j}t(X)^{j}u(X)^{\deg q -j}  \big)^{d}\cdot (f\circ (R\circ S))(X) \\
   & =   f \boxcircle (R\circ S)
   \end{align*}
   where the last equality follows from the ``relatively prime form'' of $R\circ S$ given in (\ref{compofratmap}).  The case where $\deg R=\deg q$
   produces the same result.
  \end{proof}

\begin{note} If we restrict to ${\rm PGL}_{2}(\Z )\subset\text{\sf Rat}_{\Z}$ acting on $\bast \tv{Z}_{1}$
we essentially recover the action of ${\rm PGL}_{2}(\Z)$ (by inverses of matrices) on projective classes of vectors in $\bast\Z^{2}$.
\end{note}

We now consider the effect of the action of $\text{\sf Rat}_{\Z}$ on both the ordinary and the Frobenius approximate ideal structure of $\bast \tilde{\tv{Z}}$.
Given $R\in \text{\sf Rat}_{\Z}$ define $\mathfrak{h}(R)=\max (\mathfrak{h}(p), \mathfrak{h}(q))$.  
Fix $\bast f\in \bast \tilde{\tv{Z}}^{\upmu^{\deg (R)}}_{\bast d}$.  Then
\[ \langle \mathfrak{h}(\bast f \boxcircle R) \rangle \leq \langle\mathfrak{h}(\bast f)\mathfrak{h}(R)^{\bast d}\rangle. \]
Since $\bast f \boxcircle R\in  \bast \tilde{\tv{Z}}_{\bast d\cdot\deg (R)}$, we will calculate its growth index with respect
to normalization by the degree $\bast d\cdot\deg (R)$:
\begin{align*} \upmu_{\bast d\cdot\deg (R)} (\bast f \boxcircle R) &= \langle \mathfrak{h} (\bast f \boxcircle R)^{-1/(\bast d\cdot\deg (R))} \rangle \\
& \geq  \mathfrak{h} (R)^{-1/\deg (R)}\cdot\upmu_{\bast d}(\bast f)^{1/\deg (R)}  \\
& > \upmu.
\end{align*}
Therefore, 
\[  \bast \tilde{\tv{Z}}^{\upmu^{\deg (R)}}_{\bast d} \boxcircle R\subset  \bast \tilde{\tv{Z}}^{\upmu}_{\bast d\cdot\deg (R)}.\]


Now suppose that $\bast f\in \bast\tilde{\tv{Z}}_{\upnu^{\deg (R)}}(\uptheta )_{\bast d}$
and $R(\upeta ) =\uptheta $.  Then we claim that $\bast f \boxcircle R \in \bast\tilde{\tv{Z}}_{\upnu}(\upeta )_{\bast d\cdot\deg (R)}$.  Indeed, calculating the decay
$\upnu (\bast f \boxcircle R)$ (with respect to $\upeta$)
 we have
  \begin{align*} \upnu_{\bast d\cdot\deg (R)} (\bast f \boxcircle R) & =\langle  |(\bast f \boxcircle R)(\upeta )|^{1/\bast d\cdot\deg (R)}\rangle \\
  & =  \langle  |\bast f (\uptheta) \cdot q(\upeta)|^{1/\bast d\cdot\deg (R)} \rangle \\
  &\leq \upnu.
    \end{align*}
Therefore,
\[  \bast \tilde{\tv{Z}}_{\upnu^{\deg (R)}}(\uptheta )_{\bast d} \boxcircle R\subset  \bast \tilde{\tv{Z}}_{\upnu}(\upeta )_{\bast d\cdot\deg (R)}.\]
We have shown:

\begin{theo}\label{ideoconsensusthm}  Suppose that  $\uptheta  \lessdot \upeta$.  Then 
\[  \bast \tilde{\tv{Z}}_{\upnu^{\deg (R)}}^{\upmu^{\deg (R)}}(\uptheta )_{\bast d} \boxcircle R\subset  
\bast \tilde{\tv{Z}}_{\upnu}^{\upmu}(\upeta )_{\bast d\cdot\deg (R)}\quad\text{and}\quad
 \bast \tilde{\tv{Z}}_{\bar{\upnu}}^{\bar{\upmu}}(\uptheta )_{\bast \breve{d}} \boxcircle R\subset  
 \bast \tilde{\tv{Z}}_{\bar{\upnu}}^{\bar{\upmu}}(\upeta )_{\bast \breve{d}}.
\]
In particular, if $\uptheta  \doteq \upeta$ then the associated pair of rational maps $(R,S)$ defines an approximation  consensus $ \bast \tilde{\tv{Z}}(\uptheta )\rightleftharpoons \bast \tilde{\tv{Z}}(\upeta )$.
\end{theo}

\section{Polynomial Nonvanishing Spectra and Mahler's Classification}\label{polyspectramahlersection}

For each degree class $\bast \breve{d}$, define the {\bf  $\bast \breve{\boldsymbol d}$ Frobenius nonvanishing spectrum}  by
\[  \overline{\rm Spec}[X](\uptheta )_{\bast \breve{d}} =\big\{ (\bar{\upmu}, \bar{\upnu} )\big|\;  \bast \tilde{\tv{Z}}^{\bar{\upmu}}_{\bar{\upnu}}(\uptheta)_{\bast \breve{d}}\not=0 \big\}.
\] 
The Frobenius spectrum refines as follows:
for each $\bast d\in \bast \breve{d}$, define
\[ \overline{\rm Spec}[X](\uptheta )_{\bast d} =\big\{ (\bar{\upmu}, \bar{\upnu} )\big|\;  \bast \tilde{\tv{Z}}^{\bar{\upmu}}_{\bar{\upnu}}(\uptheta)_{\bast d}\not=0 \big\} ,\]
then for $ \bast d'\in \bast\breve{d}$, $\bast d'>\bast d$, we have an inclusion
$ \overline{\rm Spec}[X](\uptheta )_{\bast d}\subset  \overline{\rm Spec}[X](\uptheta )_{\bast d'}$
and thus we obtain the filtration
\[  \overline{\rm Spec}[X](\uptheta )_{\bast \breve{d}}= \bigcup_{\bast d\in\bast\breve{d}} \overline{\rm Spec}[X](\uptheta )_{\bast d}. \]
Furthermore, if we formally add the ``northeast perimeter'' 
\begin{align}\label{northeast}\bstar\PR\R_{\leq 1}^{2} -\bstar\PR\R_{\upvarepsilon}^{2}= (\{ \bar{1}\}\times \bstar\overline{\PR\R}_{\upvarepsilon})\cup (\bstar\overline{\PR\R}_{\upvarepsilon}\times\{ \bar{1}\})\cup (\bar{1},\bar{1}) 
\end{align} to 
$\overline{\rm Spec}[X](\uptheta )_{\bast \breve{d}}$,
then for all $\bast\breve{d}<\bast \breve{e}$ there is a well-defined filtration preserving inclusion
\begin{align}\label{powerlawinclu}  \overline{\rm Spec}[X](\uptheta )_{\bast \breve{d}}\hookrightarrow \overline{\rm Spec}[X](\uptheta )_{\bast \breve{e}},\quad 
(\bar{\upmu},\bar{\upnu})\mapsto
 (\bar{\upmu}^{\bast \breve{d}/\bast \breve{e}},\bar{\upnu}^{\bast \breve{d}/\bast \breve{e}}).
 \end{align}
 The $\bast \breve{d}/\bast \breve{e}$-powers are well-defined since $\bast \breve{d}/\bast \breve{e}$ is the class of an infinitesimal {\it c.f.}\
 {\it Note} \ref{infinitesimalpower} of \S \ref{tropical}; if either of the coordinates of $(\bar{\upmu},\bar{\upnu})$ is a $\bast \breve{d}/\bast \breve{e}$-root of unity, the
 image of the pair will belong to the perimeter (\ref{northeast}).
 The {\bf Frobenius nonvanishing spectrum} is the collection
\[  \overline{\rm Spec}[X](\uptheta ) 
= \{ \overline{\rm Spec}[X](\uptheta )_{\bast \breve{d}}\}
\]
viewed as a directed system of filtered sets, where the inclusions are the power laws (\ref{powerlawinclu}).

Likewise we may define the {\bf ordinary nonvanishing spectrum} 
\[{\rm Spec}[X](\uptheta )= \{{\rm Spec}[X](\uptheta )_{\bast d}\},\quad 
{\rm Spec}[X](\uptheta )_{\bast d}= \big\{ (\upmu, \upnu )\big|\;  \bast \tilde{\tv{Z}}^{\upmu}_{\upnu}(\uptheta)_{\bast d}\not=0 \big\}.\] 

\begin{note}\label{ordinarytofrobenspectra} Observe that $(\bar{\upmu}, \bar{\upnu}) \in \overline{\rm Spec}[X](\uptheta )_{\bast d}$ implies
that for all $\upmu \in \bar{\upmu}, \upnu \in \bar{\upnu}$, $(\upmu,\upnu)\in {\rm Spec}[X](\uptheta )_{\bast d}$.  The converse
may not be true however: if $0\not=\bast f\in\bast\tilde{\tv{Z}}^{\upmu}_{\upnu}(\uptheta )_{\bast d}$ with $\bar{\upmu}_{\bast d}(\bast f)=\bar{\upmu}$
then $\bast f\not\in\bast \tilde{\tv{Z}}_{\bar{\upnu}}^{\bar{\upmu}}(\uptheta )_{\bast d}$.  So the ordinary nonvanishing spectrum is a
finer invariant.
\end{note}



We say that $\overline{\rm Spec}[X](\uptheta )$ and $\overline{\rm Spec}[X](\upeta )$
are {\bf finite power law equivalent} and write 
\[ \overline{\rm Spec}[X](\uptheta )\circeq\overline{\rm Spec}[X](\upeta )\] if there exist $m,n\in\N$ such that for all $\bast d$,
\[  \overline{\rm Spec}[X](\uptheta )_{\bast d}\subset  \overline{\rm Spec}[X](\upeta )_{\bast d\cdot m}\quad\text{and}\quad
 \overline{\rm Spec}[X](\upeta )_{\bast d}\subset  \overline{\rm Spec}[X](\uptheta )_{\bast d\cdot n}.\]  Clearly $\overline{\rm Spec}[X](\uptheta )\circeq\overline{\rm Spec}[X](\upeta )$ implies that $\overline{\rm Spec}[X](\uptheta )_{\bast \breve{d}} =  \overline{\rm Spec}[X](\upeta )_{\bast \breve{d}} $ for every class $\bast\breve{d}$ but the converse need not be true.  
 
 We have the following Corollary to Theorem \ref{ideoconsensusthm}:
\begin{coro}\label{ratspeccoro}  Suppose that  $\uptheta  \lessdot \upeta$.   Then
\[  \overline{\rm Spec}[X](\uptheta )_{\bast d}\subset  \overline{\rm Spec}[X](\upeta )_{\bast d\cdot\deg (R)} \;\;\text{and}\;\;  
\overline{\rm Spec}[X](\uptheta )_{\bast \breve{d}} \subset  \overline{\rm Spec}[X](\upeta )_{\bast \breve{d}} .\]
If $\uptheta\doteq\upeta$ then 
\[ \overline{\rm Spec}[X](\uptheta ) \circeq  \overline{\rm Spec}[X](\upeta ).\]
\end{coro}

Since finite degrees $d<e$ define the same Frobenius class, we have: 
\begin{align}\label{incprop}\overline{\rm Spec}[X](\uptheta )_{d}\subset \overline{\rm Spec}[X](\uptheta )_{e}.
\end{align}  From this we may derive the following nonvanishing
result:

\begin{prop}\label{polyslowsector}  For all $\bast d\in\bast\N$, $\{(\bar{\upmu},\bar{\upnu})|\; \bar{\upmu}<\bar{\upnu}\}\subset \overline{\rm Spec}[X](\uptheta )_{\bast d}$.
\end{prop}

\begin{proof}  We first show that $\overline{\rm Spec}[X](\uptheta )_{1}$ contains the image of ${\rm Spec}(\uptheta )$ in 
$\bstar\overline{\PR\R}_{\upvarepsilon}$. Indeed, given $\bar{\upmu}<\bar{\upnu}$, since $\bstar\overline{\PR\R}_{\upvarepsilon}$
is a dense linear order (Proposition \ref{FrobDLO}), there exists $\bar{\upmu}'$ with $\bar{\upmu}<\bar{\upmu}' <\bar{\upnu}$.
Choose $\upmu'\in\bar{\upmu}',\upnu\in\bar{\upnu}$ and $0\not=\bast n\in\bast\Z^{\upmu'}_{\upnu}(\uptheta )$. Then $\bast f_{\bast n}(X)=\bast nX-\bast n^{\perp}$
satisfies $\bar{\upmu}(\bast f_{\bast n})=[\upmu (\bast n)]\geq \bar{\upmu}'>\bar{\upmu}$ and $\bar{\upnu}(\bast f_{\bast n})=[\upnu (\bast n)]\leq \bar{\upnu}$.  
By (\ref{incprop}), this 
implies that $(\bar{\upmu},\bar{\upnu})\in\overline{\rm Spec}[X](\uptheta )_{d}$ for $d\in\N$.   For infinite degree $\bast d$, the power law inclusion
(\ref{powerlawinclu}) maps the set of pairs $\bar{\upmu}<\bar{\upnu}$ in $\overline{\rm Spec}[X](\uptheta )_{1}$ of elements which are not 
$(\bast d)^{-1}$ roots of unity onto $\{(\bar{\upmu},\bar{\upnu})|\; \bar{\upmu}<\bar{\upnu}\}$.
\end{proof}

We examine the nonvanishing spectra in terms of Mahler's classification, whose definition we now recall.
For each pair of positive integers $H$, $d$, consider the polynomial $f_{H,d}(X)$ which minimizes $|f(\uptheta )|$
amongst those $f(X)$ of degree at most $d$ 
with $f(\uptheta)\not=0$ and satisfying $\mathfrak{h}(f)\leq H$.
Define 
\[  
\mathfrak{e} (d,H) =  \frac{\log |f_{H,d}(\uptheta )|}{\log H^{-d}},
 \]
 so that  
 \[ |f_{H,d}(\uptheta )|^{1/d} = H^{-\mathfrak{e} (d,H) }.\]
Consider
the limits
\[  \mathfrak{e} (d) =  \lim_{H\rightarrow\infty}\sup\mathfrak{e} (d, H) ,\quad \mathfrak{e}= 
\lim_{d\rightarrow\infty}\sup \mathfrak{e} (d).  \]

For $d$ fixed, let $\{ g_{i}=f_{H_{i},d}\}$ be a subsequence of $\{f_{H,d}  \}_{H\geq 1}$ for which
$\mathfrak{h}(f_{H_{i},d})=H_{i}$ and $\mathfrak{e}_{i} (d) := \mathfrak{e}(d,H_{i})$
converges to $\mathfrak{e}(d)$.  The class 
\[ \bast  \widehat{g}_{d}\in \bast\tv{Z}(\uptheta )_{d}\] corresponding to
$\{ g_{i}\}$ is called a {\bf degree {\em d} best polynomial class}.  When $d=1$, then $\bast  \widehat{g}_{1}$ is identified
with a best denominator class $\bast  \widehat{q}$ via the isomorphism $\bast\Z (\uptheta )\cong \bast\tv{Z}(\uptheta )_{1} $.
Denote by $ \widehat{\bar{\upmu}}_{d}=\bar{\upmu}(\bast  \widehat{g}_{d})$ and
$ \widehat{\bar{\upnu}}_{d}=\bar{\upnu}(\bast  \widehat{g}_{d})$: the associated {\bf degree {\em d}} {\bf   best polynomial growth} and {\bf decay classes}.  

Similarly, let $\{ g_{i}= f_{H_{i},d_{i}}\}$ be a subsequence of $\{f_{H,d}  \}_{H,d\geq 1}$  
with $H_{i},d_{i}\rightarrow\infty$ for which $\mathfrak{h}(f_{H_{i},d_{i}})=H_{i}$, $\deg (f_{H_{i},d_{i}})=d_{i}$
and $\mathfrak{e}_{i} := \mathfrak{e}(d_{i},H_{i})$
converges to $\mathfrak{e}$.  Then the
class 
\[ \bast  \widehat{g}_{\bast d}\in  \bast\tilde{\tv{Z}}(\uptheta )_{\bast d} \] corresponding to $\{ g_{i}\}$
is called a {\bf degree $\boldsymbol\bast${\em d}} {\bf best polynomial
class}, with associated  {\bf degree $\boldsymbol\bast${\em d}} {\bf  best polynomial growth}  
$ \widehat{\bar{\upmu}}_{\bast d}=\bar{\upmu}_{\bast d}(\bast g)$ and 
{\bf  decay}  $ \widehat{\bar{\upnu}}_{\bast d}=\bar{\upnu}_{\bast d}(\bast g)$.   The infinite degree $\bast d=\bast\{d_{i}\}$
is called the corresponding {\bf best degree}.

\begin{note}  When $\uptheta$ is algebraic, a best polynomial class is only best amongst those polynomials which do not have 
$\uptheta$ as a root.  This contrasts with the process of producing linear best approximations for $\uptheta\in\Q$, which terminates 
in the root of the defining linear polynomial.
\end{note}

\begin{prop}\label{bestapproxpoly}  Suppose that $\uptheta$ is not algebraic of degree $\leq d$ (that $\uptheta$ is transcendental).  
Let $ ( \widehat{\bar{\upmu}}_{d},  \widehat{\bar{\upnu}}_{d})$
be a finite degree $d$ best growth-decay pair (let $( \widehat{\bar{\upmu}}_{\bast d},  \widehat{\bar{\upnu}}_{\bast d})$ be an infinite degree best growth-decay pair).
For all $\bar{\upmu} \geq  \widehat{\bar{\upmu}}_{d}$ and $\bar{\upnu} <  \widehat{\bar{\upnu}}_{d}$ 
(for all $\bar{\upmu} \geq \widehat{\bar{\upmu}}_{\bast d}$ and $\bar{\upnu} <  \widehat{\bar{\upnu}}_{\bast d}$),
\[ \bast \tv{Z}_{\bar{\upnu}}^{\bar{\upmu}}(\uptheta )_{d}= 0\quad \left( \bast \tilde{\tv{Z}}_{\bar{\upnu}}^{\bar{\upmu}}(\uptheta )_{\bast d}= 0 \right).  \]  In particular, 
\[  \overline{\rm Spec}[X](\uptheta )_{d}\subsetneq \bstar\PR\R_{\upvarepsilon}^{2}\quad \big( \overline{\rm Spec}[X](\uptheta )_{\bast d}\subsetneq \bstar\PR\R_{\upvarepsilon}^{2}\big).\]
\end{prop}

\begin{proof} We begin with the finite degree $d$ statements. Suppose otherwise: let $\bar{\upmu} \geq  \widehat{\bar{\upmu}}_{d}$ and $\bar{\upnu} <  \widehat{\bar{\upnu}}_{d}$ yet
$0\not=\bast f\in \bast \tv{Z}_{\bar{\upnu}}^{\bar{\upmu}}(\uptheta )_{ d}$.  
Then 
$\mathfrak{h}(\bast f)\cdot  \widehat{\bar{\upmu}}_{d} \leq \mathfrak{h}(\bast f)\cdot \bar{\upmu}\in\bstar\PR\R_{\upvarepsilon}$ implies
that $\mathfrak{h}(\bast f)<\mathfrak{h}(\bast  \widehat{g}_{d})$.  However, by definition of $\bast  \widehat{g}_{d}$,
we must have $\bar{\upnu} (\bast f )\geq  \widehat{\bar{\upnu}}_{d}>\bar{\upnu}$ (since by hypothesis $\bast f(\uptheta )\not=0$), contradiction. This argument applies {\it mutadis mutandis} to the infinite degree case.  
\end{proof}

Thus, we may attach to any degree $\bast d$ best polynomial class $\bast  \widehat{g}_{\bast d}$ 
a horizontal strip in the $(\bar{\upmu},\bar{\upnu})$ plane:
\[ V(\bast  \widehat{g}_{\bast d}) = \big\{ (\bar{\upmu},\bar{\upnu}) \big|\; \bar{\upmu} \geq  \widehat{\bar{\upmu}}_{d}\text{ and }\bar{\upnu} <  \widehat{\bar{\upnu}}_{d}  \big\} \]
parametrizing
a region where $\bast\tilde{\tv Z}^{\bar{\upmu}}_{\bar{\upnu}}(\uptheta)_{\bast d} $ vanishes.  
The extent to which $V(\bast  \widehat{g}_{\bast d})$ persists in $\bast\tilde{\tv Z}^{\bar{\upmu}}_{\bar{\upnu}}(\uptheta)_{\bast \breve{d}} $
is a function of the Mahler type of $\uptheta$, a point which will be discussed in detail at the end of this section.

The Frobenius polynomial spectrum $ \overline{\rm Spec}[X](\uptheta )$ picks out the real algebraic
numbers in the same way that the ordinary linear spectrum ${\rm Spec}(\uptheta )$ picks out the rationals:

\begin{coro} The following statements are equivalent.
\begin{enumerate}
\item[i.]  $\uptheta\in\R$ is algebraic.
\item[ii.] $ \overline{\rm Spec}[X](\uptheta )_{d}=\bstar\PR\R_{\upvarepsilon}^{2}$ for some $d\in\N$.
\item[iii.] $ \overline{\rm Spec}[X](\uptheta )_{\bast d}=\bstar\PR\R_{\upvarepsilon}^{2}$ for every $\bast d\in\bast\N -\N$.
\end{enumerate}
\end{coro}

\begin{proof}  i.\ $\Rightarrow$ ii., iii.\ If there exists $f(X)\in\Z[X]$ (a standard polynomial) with $f(\uptheta )=0$ then 
for all $\bast d\geq \deg (f)$, $\mathfrak{h}_{\bast d}(f)\cdot \upmu\in\bstar\PR\R_{\upvarepsilon}$ for all $\upmu\in \bstar\PR\R_{\upvarepsilon}$, and $\upnu_{\bast d} (f)=-\infty\leq\upnu$ for all $\upnu\in \bstar\PR\R_{\upvarepsilon}$.  Then for such $\bast d$, $ \overline{\rm Spec}[X](\uptheta )_{\bast d}=\bstar\PR\R_{\upvarepsilon}^{2}$.
ii., iii.\ $\Rightarrow$ i.\ Follows from Proposition \ref{bestapproxpoly}.
\end{proof}

Denote by $\mathfrak{d}$ the smallest degree $d$ for which $\mathfrak{e}  (d)$ is infinite; if
$\mathfrak{e}  (d)$ is finite for all $d$ write $\mathfrak{d}  =\infty$. 
Recall that Mahler's classification \cite{Ba} consists of the identification of
the following non-empty, mutually exclusive classes of numbers whose union is $\R$:  $\uptheta$ is an

\begin{itemize}
\item[-]  {\bf $\boldsymbol A$-number} if $\mathfrak{e}  =0$ and $\mathfrak{d}  = \infty$.  The 
$A$-numbers are exactly the real algebraic numbers.
\item[-]   {\bf $\boldsymbol S$-number}if $0<\mathfrak{e} <\infty$ and $\mathfrak{d}  = \infty$.
Lebesgue-almost all real numbers are $S$-numbers.
\item[-]   {\bf $\boldsymbol T$-number} if $\mathfrak{e}  =\infty$ and $\mathfrak{d}  = \infty$.
\item[-]  {\bf $\boldsymbol U$-number} if $\mathfrak{e}  =\infty$ and $\mathfrak{d}  < \infty$.
\end{itemize}
Note that if $\uptheta$ and $\upeta$ are algebraically dependent, they are of the same Mahler class \cite{Bu}: in particular,
the relation $ \uptheta\doteq\upeta$ preserves Mahler class.

We now interpret Mahler's classification in terms of the images of pairs of best polynomial growth and best polynomial
decay in the Frobenius growth-decay semi-ring $\bstar\overline{\PR\R}_{\upvarepsilon}$.

\begin{theo}\label{Atheo}  $\uptheta$ is an $A$-number $\Leftrightarrow$ for every infinite degree polynomial best growth decay pair $( \widehat{\bar{\upmu}}_{\bast d},  \widehat{\bar{\upnu}}_{\bast d})$,
\[  \widehat{\bar{\upmu}}_{\bast d}< ( \widehat{\bar{\upnu}}_{\bast d})^{1/\bast d}.\]
\end{theo}

\begin{proof} $\uptheta$ is an $A$-number $\Leftrightarrow$ for every infinite degree best polynomial class $\bast  \widehat{g}_{\bast d}$,
\[ \big|\bast  \widehat{g}(\uptheta )\big|^{1/\bast d} = \big(\mathfrak{h}(\bast  \widehat{g})^{-1}\big)^{\bast \mathfrak{e}} \]
where $\bast\mathfrak{e} = \bast \{\mathfrak{e}_{i}\}$ is infinitesimal.  Then 
$\big(\mathfrak{h}(\bast  \widehat{g})^{-1}\big)^{\bast \mathfrak{e}}>\big(\mathfrak{h}(\bast  \widehat{g})^{-1}\big)^{e} $ for any $e\in\R_{+}$,
and the result follows on passing to $\bstar\overline{\PR\R}_{\upvarepsilon}$ classes of degree $\bast d$ normalized growth and decay.
\end{proof}


\begin{theo}\label{S}  $\uptheta$ is an $S$-number $\Leftrightarrow$ for every infinite degree polynomial best growth decay pair  
$( \widehat{\bar{\upmu}}_{\bast d},  \widehat{\bar{\upnu}}_{\bast d})$,
\[  \widehat{\bar{\upmu}}_{\bast d}= ( \widehat{\bar{\upnu}}_{\bast d})^{1/\bast d}.\]
In this case, the associated vanishing strip $V(\bast  \widehat{g}_{\bast d})$ persists
in $  \overline{\rm Spec}[X](\uptheta )_{\bast \breve{d}}$: 
\[    \overline{\rm Spec}[X](\uptheta )_{\bast \breve{d}}\cap V(\bast  \widehat{g}_{\bast d})=\emptyset .\]
In particular, for any infinite best degree $\bast d$,
\[ \overline{\rm Spec}[X](\uptheta )_{\bast \breve{d}}\not= \bstar\overline{\PR\R}_{\upvarepsilon}^{2}.\]
\end{theo}

\begin{proof} $\uptheta$ is an $S$-number $\Leftrightarrow$ for every 
infinite degree best polynomial class $\bast  \widehat{g}$ of degree $\bast d$
\[ |\bast  \widehat{g}(\uptheta )|^{1/\bast d} = (\mathfrak{h}(\bast  \widehat{g})^{-1})^{\bast \mathfrak{e}} \]
where $\bast \mathfrak{e}\simeq\mathfrak{e}\in\R_{+}$.      
Then the first statement follows upon passage to $\bstar\overline{\PR\R}_{\upvarepsilon}$ classes of degree $\bast d$ normalized growth and decay.
Now suppose that $\bast d<\bast e$ both belong to the class $\bast \breve{d}$ i.e.\ $\bast e/\bast d$ is bounded.  Suppose further that there
exists a best growth decay pair $( \widehat{\bar{\upmu}}_{\bast e},  \widehat{\bar{\upnu}}_{\bast e})\in V(\bast  \widehat{g}_{\bast d})$.  Then $ \widehat{\bar{\upmu}}_{\bast e}\geq  \widehat{\bar{\upmu}}_{\bast d}$
and $ \widehat{\bar{\upnu}}_{\bast e}< \widehat{\bar{\upnu}}_{\bast d}$, however:
\[  ( \widehat{\bar{\upnu}}_{\bast d})^{1/\bast d}=   \widehat{\bar{\upmu}}_{\bast d}\leq  \widehat{\bar{\upmu}}_{\bast e}=( \widehat{\bar{\upnu}}_{\bast e})^{1/\bast e} \]
that is
\[   ( \widehat{\bar{\upnu}}_{\bast d})^{\bast e/\bast d}= \widehat{\bar{\upnu}}_{\bast d}\leq   \widehat{\bar{\upnu}}_{\bast e}, \]
contradiction.

\end{proof}


\begin{theo}  $\uptheta$ is a $T$-number $\Leftrightarrow$ for every 
\begin{itemize} 
\item[-] infinite degree best growth decay pair $( \widehat{\bar{\upmu}}_{\bast d},  \widehat{\bar{\upnu}}_{\bast d})$,
\[  \widehat{\bar{\upmu}}_{\bast d}> ( \widehat{\bar{\upnu}}_{\bast d})^{1/\bast d}.\]
\item[-] finite degree best growth decay pair $( \widehat{\bar{\upmu}}_{d},  \widehat{\bar{\upnu}}_{d})$, 
\[   \widehat{\bar{\upmu}}_{d}=  \widehat{\bar{\upnu}}_{d}.\]
\end{itemize}
In the latter case, the associated vanishing strip $V(\bast  \widehat{g}_{d})$ persists
in $  \overline{\rm Spec}[X](\uptheta )_{\breve{1}}=\overline{\rm Spec}[X](\uptheta )_{\breve{d}}$: 
\[    \overline{\rm Spec}[X](\uptheta )_{\breve{1}}\cap V(\bast  \widehat{g}_{d})=\emptyset .\]
In particular, 
\[ \overline{\rm Spec}[X](\uptheta )_{ \breve{1}}\not= \bstar\overline{\PR\R}_{\upvarepsilon}^{2}.\]

\end{theo}

\begin{proof}  $\uptheta$ is a $T$-number $\Leftrightarrow$ $\mathfrak{e}(d)<\infty$ for all $d$, the set of which is unbounded.  
Let $\bast  \widehat{g}_{\bast d}$ be an infinite degree $\bast d$ polynomial best class defined by the sequence $\{g_{i}= f_{H_{i},d_{i}}\}$.   
Let $\bast d=\bast \{ d_{i}\}$ and $\bast\mathfrak{e}=\bast\{ \mathfrak{e}(H_{i},d_{i})\}$, the latter infinite by hypothesis. Then we have 
$|\bast  \widehat{g}_{\bast d}(\uptheta )|^{1/\bast d} = (\mathfrak{h}(\bast  \widehat{g}_{\bast d})^{-1})^{\bast\mathfrak{e}}$ which implies
after passing to Frobenius classes that
\[   \widehat{\bar{\upmu}}_{\bast d}=( \widehat{\bar{\upnu}}_{\bast d})^{1/(\bast d\bast \mathfrak{e})}>( \widehat{\bar{\upnu}}_{\bast d})^{1/\bast d}. \]
If $\bast  \widehat{g}_{d}$ is a finite degree $d$ polynomial best class, then since $0\leq \mathfrak{e}(d)<\infty$, we have
\[   \widehat{\bar{\upmu}}_{d} \leq   \widehat{\bar{\upnu}}_{d} . \]
But if the inequality were strict, it would contradict Proposition \ref{polyslowsector}.
The statement concerning vanishing strips,
is argued exactly as in Theorem \ref{S}.
\end{proof}

We leave the proof of the last Theorem in this series to the reader.

\begin{theo}\label{Utheo}  $\uptheta$ is a $U$-number $\Leftrightarrow$  for every 
\begin{itemize} 
\item[-] infinite degree best growth decay pair $( \widehat{\bar{\upmu}}_{\bast d},  \widehat{\bar{\upnu}}_{\bast d})$,
\[  \widehat{\bar{\upmu}}_{\bast d}> ( \widehat{\bar{\upnu}}_{\bast d})^{1/\bast d}.\]
\item[-] There exists $\mathfrak{d}\in\N$ such that for every finite degree $d\geq\mathfrak{d}$ polynomial best growth decay pair 
$( \widehat{\bar{\upmu}}_{d},  \widehat{\bar{\upnu}}_{d})$, 
\[   \widehat{\bar{\upmu}}_{d}>  \widehat{\bar{\upnu}}_{d}.\]
\end{itemize}
\end{theo}




\begin{figure}[htbp]
\centering
\includegraphics[width=5in]{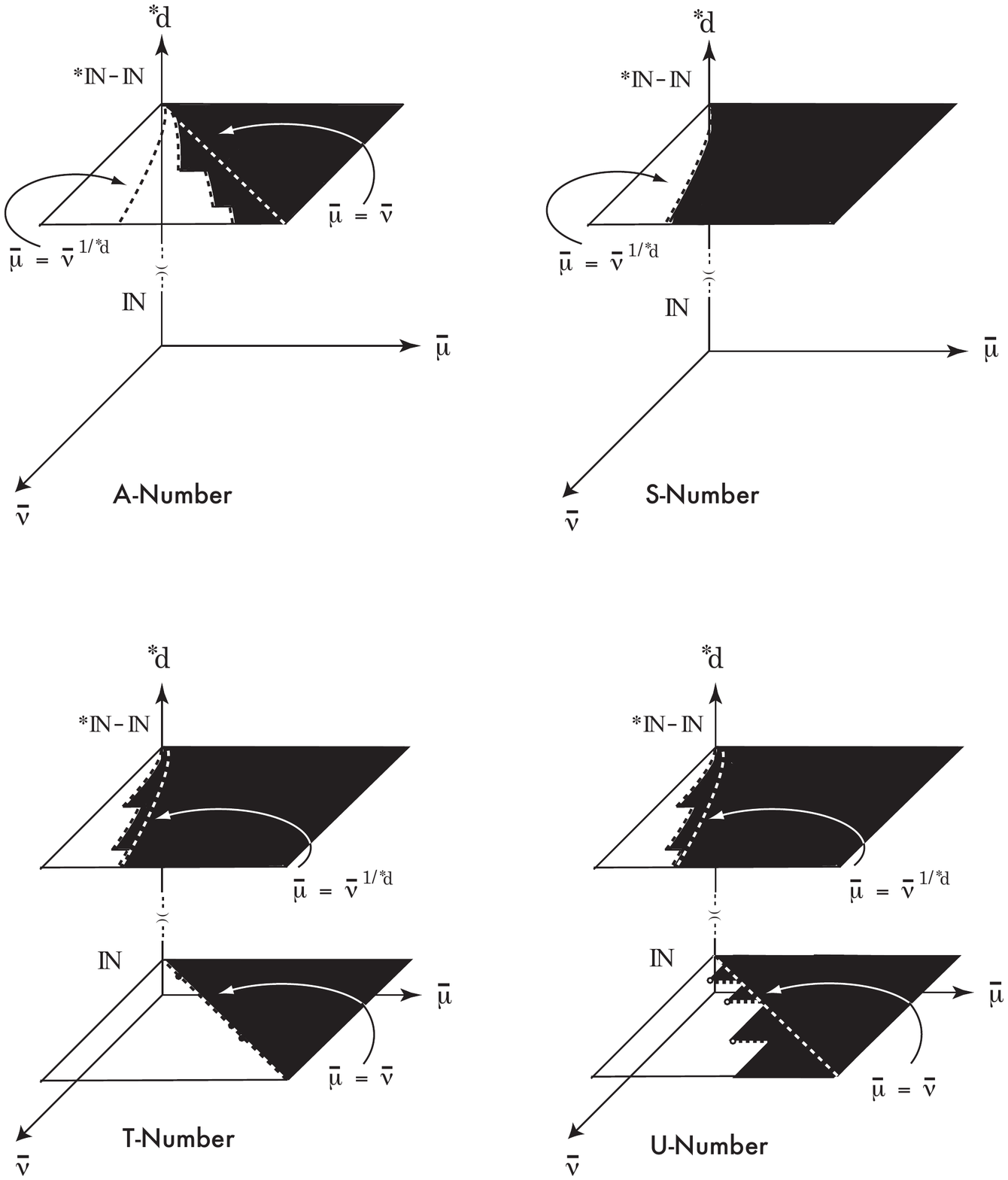}
\caption{Frobenius Nonvanishing Spectra According to Mahler Class}\label{Mahlerportraits}
\end{figure}

In Figure \ref{Mahlerportraits} we have displayed the sheets of the filtered Frobenius nonvanishing spectra for each of the four Mahler classes, (in the $A$-number
portrait we have removed spectral elements corresponding to groups all of whose elements are multiples of the minimal polynomial).

Recall \cite{Bu} that each of the transcendental Mahler classes can be further partitioned according to the {\bf Mahler type}.  To observe the type,
we must work with the ordinary nonvanishing spectra.

For $S$-numbers,
the type $\mathfrak{t}$ is equal to the exponent $\mathfrak{e}\in [1,\infty)$. 

\begin{theo}\label{SType}  Let $\uptheta$ be an $S$-number of type $\mathfrak{t}\in [1,\infty)$.
Then for every infinite degree $\bast d$ best polynomial class 
$\bast \widehat{g}_{\bast d}$,
the associated ordinary growth-decay pair satisfies
\[   \widehat{\upmu}_{\bast d}=   ( \widehat{\upnu}_{\bast d})^{1/(\bast \mathfrak{t}\bast d)}>
( \widehat{\upnu}_{\bast d})^{1/(\mathfrak{l}\bast d)} \]
where $\bast \mathfrak{t}=\bast \{\mathfrak{e}(d_{i})\}$ and where $\mathfrak{l}<\mathfrak{t}$.
\end{theo}

\begin{proof}   Clear from the definition of best polynomial class and the type $\mathfrak{t}$.
\end{proof}

 For $T$ numbers, one writes for each finite degree $d$
\[  \mathfrak{e}(d) = d^{\mathfrak{t}_{d}-1} \]
and then the type is defined
\[ \mathfrak{t} = \lim\sup \mathfrak{t}_{d}\in [1,\infty].\]

\begin{theo}  Let $\uptheta$ be a $T$-number of type $\mathfrak{t}\in [1,\infty]$.  Then for every infinite degree $\bast d$ best polynomial class $\bast \widehat{g}_{\bast d}$,
the associated ordinary growth-decay pair satisfies
\[    \widehat{\upmu}_{\bast d}=  (  \widehat{\upnu}_{\bast d})^{1/\bast d^{\bast \mathfrak{t}}} >(  \widehat{\upnu}_{\bast d})^{1/\bast d^{\mathfrak{l}}}\]
where $\bast \mathfrak{t}=\bast \{\mathfrak{t}_{d_{i}}\}$ and where $\mathfrak{l}<\mathfrak{t}$.
\end{theo}

\begin{proof}   From the definitions:
\[ |\bast \widehat{g}_{\bast d}(\uptheta ) |^{1/\bast d}
=(\mathfrak{h}(\bast \widehat{g}_{\bast d})^{-1})^{\bast d^{\bast\mathfrak{t}-1} } =(\mathfrak{h}(\bast \widehat{g}_{\bast d})^{-1/\bast d})^{\bast d^{\bast\mathfrak{t}}}.
 \]
 Taking $\bast d^{\bast\mathfrak{t}}$-roots of both sides and passing to $\bstar\PR\R_{\upvarepsilon}$ classes gives the result, since $\bast\mathfrak{t}>\mathfrak{l}$.
\end{proof}

Finally, the type of a $U$ number is the first integer $\mathfrak{t}$ for which $\mathfrak{e}(\mathfrak{t})=\infty$.  We leave the proof of the following to the reader:

\begin{theo}\label{UType}  Let $\uptheta$ be an $U$-number of type $\mathfrak{t}\in \N$.  Then for every $\bast\mathfrak{t}$ and $\bast d\geq \mathfrak{t}$, there exists a best polynomial class $\bast \widehat{g}_{\bast d}$ such that
the associated ordinary growth-decay pair satisfies
\[    \widehat{\upmu}_{\bast d}>  (  \widehat{\upnu}_{\bast d})^{1/\bast d^{\bast \mathfrak{t}}} .\]
\end{theo}

Thus, we can say that the $S$, $T$ and $U$ numbers are those for which the decay is related to growth linearly, polynomially and exponentially, respectively, as a function of degree.

\section{Resultant Arithmetic}\label{resultantarithmetic}

In this section we present polynomial diophantine approximation as a natural nonlinear extension of classical diophantine approximation 
(approximation by linear polynomials). 
We are faced first with the problem of finding the right notions of sum and product in the polynomial ring $\Z[X]$.
By ``right'', we mean that
\begin{itemize}
\item[I.] They should be compatible with fractional sum and product in $\Q$ i.e. 
the map $a/b\mapsto bX-a $ should be a monomorphism modulo multiplication by non zero integer constants.  As a consequence, this will ensure that
the map,
\[ \bast\Z (\uptheta )\rightarrow\bast\tv{Z}(\uptheta )_{1}
\subset \bast\tv{Z}(\uptheta ), \quad \bast n\mapsto \bast n X-\bast n^{\perp},\] will respect fractional sum and product of numerator denominator pairs of diophantine approximations.
\item[II.] If $\upalpha,\upbeta$ are algebraic and $f(\upalpha )=0$, $g(\upbeta )=0$ then $\upalpha+\upbeta$
should be a root of the sum of $f$ and $g$, and $\upalpha\upeta$ should be a root
of the product of $f$ and $g$.
\end{itemize}

Regrettably, neither the Cauchy nor the Dirichlet products satisfy the above criteria.  

We begin by defining the sought after product.  In what follows, denote by 
\[ \tv{Z} :=\{  f\in\Z [X]|\; \deg f\geq 1\} \cup \{ 1\},\] 
a monoid with respect to the Cauchy product.    Consider elements $f,g\in \tv{Z}$:
\[  f(X)=a_{m}X^{m}+\cdots + a_{0} = a_{m}\prod_{i=1}^{m} (X-\upalpha_{i}),\quad
g(X)=b_{n}X^{n}+\cdots + b_{0} = b_{n}\prod_{j=1}^{n} (X-\upbeta_{j}) .\]  If $\deg (f),\deg(g)\geq 1$, their {\bf resultant product}
is defined
\[    (f  \boxtimes g)(X) := a_{m}^{n}b_{n}^{m}\prod  (X-\upalpha_{i}\upbeta_{j}) ;\]
otherwise, if either $f$ or $g$ is $1$, it is defined to be $1$.  The resultant product was first defined for monic polynomials \cite{BC} and 
later for general polynomials \cite{Gl}, where it is referred to as the {\it tensor product} of polynomials.
Note that $\deg (f  \boxtimes g)=mn=\deg(f)\deg(g)$.
The resultant product is clearly commutative and associative, and the polynomial $1_{ \boxtimes}(X):=X-1$ acts as the identity.
The proof of the Theorem which follows was suggested to us by Gregor Weingart.

\begin{theo}\label{resulininteral}  If $f,g\in\tv{Z}$ then $f\boxtimes g\in\tv{Z}$. 
\end{theo}

\begin{proof} We assume that $f,g\not=1$. Consider the homogenization of $f$ with respect to the new variable $Y$:
\[  f_{X}(Y)=Y^{m}f(X/Y)= 
\sum_{i=0}^{m} \big(a_{m-i}X^{m-i}\big)Y^{i},
\]
viewed as a polynomial in $\big(\Z[X]\big)[Y]$ i.e.\ in the single variable $Y$.
Then 
\[  (f\boxtimes g)(X) = (-1)^{mn} a_{0}^{n}b_{n}^{m}\prod \big((X/\upalpha_{i})-\upbeta_{j}\big)= (-1)^{mn}{\rm res}\big(f_{X}(Y), g(Y)\big) ,  \]
where ${\rm res}(\cdot ,\cdot)$ denotes the resultant.  But the resultant is the determinant of a corresponding Sylvester matrix \cite{Ak},
whose non zero entries consist of the coefficients of $f_{X}(Y)$ and $g(Y)$.  Thus $ f\boxtimes g\in\tv{Z}$ as claimed.

\end{proof}

The {\bf resultant sum} of $f,g\in\tv{Z}$, is defined
\[(f  \boxplus g)(X) := a_{m}^{n}b_{n}^{m}\prod  \big(X-(\upalpha_{i}+\upbeta_{j})\big)\]
when $f,g\not=1$; otherwise it is defined to be $1$.
The resultant sum defines an element of $\tv{Z}$ by an argument similar to that found in Theorem \ref{resulininteral}.
As in the case of the resultant product, $\deg (f  \boxplus g)=\deg(f)\deg(g)$.
The identity element for $\boxplus$ is the polynomial $1_{\boxplus}(X) :=X$.
There is also a {\bf resultant difference} 
\[ f \boxminus g  := a_{m}^{n}b_{n}^{m}\prod  \big(X-(\upalpha_{i}-\upbeta_{j})\big).\]

\begin{prop}\label{distribute} Both the resultant product and the resultant sum distribute over the Cauchy product:
\[  f\boxtimes (g\cdot h) =   (f\boxtimes g)\cdot (f\boxtimes h), \quad
f\boxplus (g\cdot h) =   (f\boxplus g)\cdot (f\boxplus h)
.\]
In particular, $\tv{Z}$ has the structure of a double semiring (a semiring with respect to
each of $\boxplus,\boxtimes$ separately).
\end{prop}

\begin{proof}  Consider the polynomials $f(X)=a_{m}\prod_{i=1}^{m} (X-\upalpha_{i})$,
 $g(X)=b_{n}\prod_{j=1}^{n} (X-\upbeta_{j})$ and $h(X) = c_{p}\prod_{k=1}^{p} (X-\upgamma_{k})$.  Then
 \begin{align*}  
 f\boxtimes (g\cdot h) & = \bigg\{ a_{m}\prod_{i=1}^{m}(X-\upalpha_{i})\bigg\} \boxtimes \bigg\{ b_{n}\prod_{j=1}^{n} (X-\upbeta_{j})\cdot c_{p}\prod_{k=1}^{p} (X-\upgamma_{k})\bigg\} \\
  & = a_{m}^{n+p}(b_{n}c_{p})^{m}\prod ( X-\upalpha_{i}\upbeta_{j})\cdot \prod ( X-\upalpha_{i}\upgamma_{k}) \\
  & \\
  & = (f\boxtimes g)\cdot (f\boxtimes h).
 \end{align*}
 An identical argument shows that the resultant sum distributes over the Cauchy product.  
\end{proof}

For the simple reason of degree, the resultant product does not distribute over the resultant sum: 
\begin{align*} \deg(f\boxtimes (g\boxplus h)) & =\deg (f)\deg(g)\deg(h) \\ 
 & \leq  \deg (f)^{2}\deg(g)\deg(h)  \\
&= \deg((f\boxtimes g)\boxplus (f\boxtimes h)).
\end{align*}
  Nevertheless, 
the subset of linear polynomials $\tv{Z}_{1}$ is closed with respect to $\boxplus$, $\boxtimes$, and one may verify
that for elements of  $\tv{Z}_{1}$, $f\boxtimes (g\boxplus h)=(f\boxtimes g)\boxplus (f\boxtimes h)$.  Thus $\tv{Z}_{1}$ 
is a semiring with respect to the operations $\boxplus, \boxtimes$.  We regard $\tv{Z}$ as a double monoid with
respect to $\boxplus, \boxtimes$. 

Let 
\[ \check{\tv{Z}} := \tv{Z}/\sim\]
where $f\sim f'$ $\Leftrightarrow$ there exists $0\not= n\in\Z$ such that $f=nf'$.  Then the operations $\boxplus,\boxtimes$
respect $\sim$ and define operations in $\check{\tv{Z}}$ making the latter a double monoid as well.  Denote by  $\check{\tv{Z}}_{1}$
the classes of linear polynomials,  which is by the above remarks a semiring with respect to $\boxplus,\boxtimes$. 

\begin{prop}\label{isowithQ} Consider the semiring $\tilde{\Q}=\{(a,b)\in\Z^{2}|\; b\not=0\}$ equipped with the usual fractional laws of addition and multiplication.
 The  bijective map
\[ \tilde{\Q}\hookrightarrow\tv{Z}_{1},\quad a/b\mapsto bX-a\]
induces an isomorphism of fields $\Q\cong\check{\tv{Z}}_{1}$ in which $\times,+,-$ are taken to $\boxtimes,\boxplus, \boxminus$.
\end{prop}

\begin{proof} The map $(a,b)\mapsto bX-a$ clearly identifies isomorphically $\tilde{\Q}=\{(a,b)|\; b\not=0\}$ with $\tv{Z}_{1}$.
Quotienting by the multiplicative
action of $\Z-\{ 0\}$ on either side gives the result.
\end{proof}

\begin{prop}  The map 
\[\tv{Z}_{1}\times\tv{Z}\longrightarrow\tv{Z},\quad (l(X),f(X))\mapsto l(X)\boxtimes f(X)\] defines an action of $\Q\cong\check{\tv{Z}}_{1}$ by $\boxplus$-isomorphisms of $ \check{\tv{Z}}$.
\end{prop}

\begin{proof}  The action is clearly well-defined, by isomorphisms since the elements of $\check{\tv{Z}}_{1}$
are $\boxtimes$-invertible.
\end{proof}


\begin{note} We may regard $\check{\tv{Z}}$ as a generalized
field extension of $\Q$.   Since $\boxtimes$ does not distribute over $\boxplus$, it is closely related to the notion of a {\bf nonlinear number field}:  an extension of $\Q$
defined using a projectivization of the field algebra of $\Q$ with the operations of Cauchy and Dirichlet products.  See \cite{GeVe}.
\end{note}

The (classes of) monic
polynomials in $\tv{Z}$ play the role of integers, the {\bf resultant integers}, which
are evidently closed w.r.t.  $\boxplus$, $ \boxminus$, $\boxtimes$ and Cauchy product.  We denote them $\tv{O}$, and note
that $\Z$ is identified via the isomorphism of Proposition \ref{isowithQ} with $\check{\tv{O}}_{1} := 
\tv{O}_{1}/\sim$.

The resultant sum, difference and product extend in the obvious way to 
$ \bast\tilde{\tv{Z}}$ e.g.\
$\bast f\boxplus\bast g= \bast \{f_{i}\boxplus g_{i}\}$.  
The next result shows that the normalized growth is supermultiplicative with respect to resultant sums, differences and products.

\begin{prop}\label{Mahlersubresultmult} Let $\bast f\in\bast\tilde{\tv{Z}}_{\bast d}$,$\bast g\in\bast\tilde{\tv{Z}}_{\bast e}$.   Then
\begin{align}\label{MahlSupMult} \upmu_{\bast d}(\bast f)\cdot\upmu_{\bast e}(\bast g)\;\;\leq \;\; \upmu_{\bast d\cdot \bast e}(\bast f\boxplus\bast g),\;\upmu_{\bast d\cdot \bast e}(\bast f \boxminus\bast g),\; \upmu_{\bast d\cdot \bast e}(\bast f\boxtimes\bast g).
\end{align}
\end{prop}

\begin{proof} We begin with $\boxtimes$ and prove
$ \mathfrak{m}_{\bast d\cdot\bast e}(\bast f\boxtimes\bast g) \leq \mathfrak{m}_{\bast d}(\bast f)\cdot \mathfrak{m}_{\bast e}(\bast g)$.
It suffices to show that for the non normalized Mahler measure,
\[\mathfrak{m}(\bast f\boxtimes\bast g)\leq \mathfrak{m}(\bast f)^{\deg (\bast g)}\cdot \mathfrak{m}(\bast g)^{\deg (\bast f)}, \]
and moreover, it is enough to prove this for standard polynomials $f,g$ of degree $m$, $n$ respectively: 
\begin{align*} \mathfrak{m}(f\boxtimes g) & = a_{m}^{n}b_{n}^{m}\prod \max (1,|\upalpha_{i}\upbeta_{j}|) \\
& \leq a_{m}^{n}\prod_{i=1}^{m}\bigg(\max (1,|\upalpha_{i}|)^{n} \big(b_{n}\prod_{j=1}^{n}  \max (1,|\upbeta_{j}|)\big)\bigg)  \\
& = \mathfrak{m}(f)^{n}\cdot \mathfrak{m}(g)^{m}.
\end{align*}
As for $\boxplus$:
\begin{align*} \mathfrak{m}(f\boxplus g) & =a_{m}^{n}b_{n}^{m}\prod \max (1,|\upalpha_{i}+\upbeta_{j}|) \\
& \leq a_{m}^{n}\prod_{i=1}^{m} \left(b_{n}  \prod_{j=1}^{n}\left(  \max (1,|\upalpha_{i}|) + \max (1,|\upbeta_{j}|) \right)\right) \\
& =  a_{m}^{n}\prod_{i=1}^{m} \left( \max (1,|\upalpha_{i}|)^{n}\left(b_{n}  \prod_{j=1}^{n}\bigg( 1 + \max (1,|\upbeta_{j}|)\cdot \max (1,|\upalpha_{i}|)^{-1} \bigg)\right)\right)\\
& \leq a_{m}^{n}\prod_{i=1}^{m}\left(  \max (1,|\upalpha_{i}|)^{n}\left(b_{n}  \prod_{j=1}^{n}\left( 1 + \max (1,|\upbeta_{j}|)  \right)\right)\right)\\
& \leq a_{m}^{n}\prod_{i=1}^{m} \left( \max (1,|\upalpha_{i}|)^{n}\left(b_{n}  \prod_{j=1}^{n}2 \max (1,|\upbeta_{j}|)  \right)\right) \\
& = 2^{mn}\mathfrak{m}(f)^{n}\cdot\mathfrak{m}(g)^{m}.
\end{align*}
Applying the above calculation to $\bast f\in\bast\tilde{\tv{Z}}_{\bast d}$,$\bast g\in\bast\tilde{\tv{Z}}_{\bast e}$ gives
\[ \mathfrak{m}_{\bast d\bast e}(\bast f\boxplus \bast g)\leq  2\mathfrak{m}_{\bast d}(\bast f)\cdot\mathfrak{m}_{\bast e}(\bast g).\]
The inequality (\ref{MahlSupMult}) for $\boxplus$ follows immediately on passage to growth classes.  The corresponding $ \boxminus$ inequality is proved in much the same way.
\end{proof}


\section{Polynomial Approximate Ideal Arithmetic I: Finite Degree}\label{polygdasec}

 Given
$\bast f\in \bast\tilde{\tv{Z}}(\uptheta )$ and $\bast g\in\bast\tilde{\tv{Z}}(\upeta )$, in this section we look for conditions which ensure that 
\[  \bast f\boxtimes\bast g\in  \bast\tilde{\tv{Z}}(\uptheta\upeta ),\quad \bast f\boxplus\bast g\in\bast\tilde{\tv{Z}} [X](\uptheta+\upeta ),\quad  \bast f \boxminus\bast g\in\bast\tilde{\tv{Z}} [X](\uptheta-\upeta ).\]
In other words, we seek to formulate polynomial approximate ideal composition laws that are compatible with the linear approximate ideal arithmetic viz. (\ref{gdprodintro}).

It is worth recording the following ``ideal-theoretic'' arithmetic for elements of $\bar{\Q}\cap\R$.
Given $\upalpha\in\bar{\Q}\cap\R$, let $\llparenthesis\upalpha \rrparenthesis$ denote the Cauchy semigroup ideal of $(\bast\tilde{\tv{Z}},\cdot) $
generated by the minimal polynomial
$m_{\upalpha}(X)$ of $\upalpha$: that is, 
\[ \llparenthesis\upalpha \rrparenthesis:=\{ \bast f\in\bast\tilde{\tv{Z}}|\; \bast f(\upalpha )=0\}.\]   Note that $\llparenthesis\upalpha\rrparenthesis\subsetneqq \bast\tilde{\tv{Z}}(\upalpha )$.

\begin{prop}\label{simplegda}  For all $\upalpha,\upbeta\in\bar{\Q}\cap\R$, 
\[\llparenthesis\upalpha\rrparenthesis\boxtimes\llparenthesis\upbeta\rrparenthesis \subset \llparenthesis\upalpha\upbeta\rrparenthesis,\quad
 \llparenthesis\upalpha \rrparenthesis\boxplus\llparenthesis\upbeta \rrparenthesis \subset \llparenthesis\upalpha+\upbeta\rrparenthesis, \quad
\llparenthesis\upalpha \rrparenthesis \boxminus\llparenthesis\upbeta \rrparenthesis \subset \llparenthesis\upalpha-\upbeta\rrparenthesis .\]
\end{prop}

\begin{proof}  Trivial.
\end{proof}

Consider the Frobenius growth approximate ideal on $\llparenthesis\upalpha \rrparenthesis$
\[ \llparenthesis\upalpha \rrparenthesis^{\bar{\upmu}}_{\bast d}= 
 \llparenthesis\upalpha \rrparenthesis\cap \bast\tilde{\tv{Z}}^{\bar{\upmu}}(\upalpha )_{\bast d}.
\]
There is no interesting decay filtration on $ \llparenthesis\upalpha \rrparenthesis$ since for all $\bast f\in \llparenthesis\upalpha \rrparenthesis$, $\upnu (\bast f)=-\infty$.

\begin{prop} Let $\upalpha\in\bar{\Q}\cap\R$. If $\bar{\upmu}\geq \bar{\upnu}^{1/\bast d}$ then 
\begin{align}\label{exactalgequal}
 \bast\tilde{\tv{Z}}^{\bar{\upmu}}_{\bar{\upnu}}(\upalpha )_{\bast d}=\llparenthesis\upalpha \rrparenthesis^{\bar{\upmu}}_{\bast d}.
 \end{align}
\end{prop}

\begin{proof}  Suppose that there exists $\bast f\in \bast\tilde{\tv{Z}}^{\bar{\upmu}}_{\bar{\upnu}}(\upalpha )_{\bast d}-\llparenthesis\upalpha \rrparenthesis^{\bar{\upmu}}_{\bast d}$, which we may assume, without loss of generality, is a best class.   Then
\[ \bar{\upmu}_{\bast d}(\bast f) >\bar{\upmu}\geq \bar{\upnu}^{1/\bast d} \geq \bar{\upnu}_{\bast d}(\bast f)^{1/\bast d}\]
which contradicts the result of Theorem \ref{Atheo}.
\end{proof}

We will call the triple $(\bar{\upmu},\bar{\upnu},\bast d)$ {\bf {\em A}-exact} if (\ref{exactalgequal}) holds: otherwise we call it
{\bf {\em A}-approximative}.

We will now describe $\boxtimes$-approximate ideal arithmetic for arbitrary real numbers, a discussion which will comprise the rest of this paper.  All of the statements proved here have counterparts for
$\boxplus, \boxminus$ whose formulations and proofs can be produced, {\it mutatis mutandis}.
In this section, we restrict ourselves to finite degree polynomials i.e.\ elements of $\bast \tv{Z}$.  The infinite degree case presents
additional complications which will be deferred to \S \ref{polygdasec2}.  We develop first the ordinary approximate ideal arithmetic, from which the Frobenius approximate ideal arithmetic follows as an immediate corollary.

Let $\uptheta\in\R$ and suppose that $\bast f\in\bast\tv{Z}(\uptheta )_{d}$.  
Denote
by $\mathbbm{r}(\bast f)$ the set of $\deg \bast f\leq d$ roots (repeated according to multiplicity).  
Then there exist $\bast \upalpha_{i_{k}}\in \mathbbm{r}(\bast f)$, $k=1,\dots l$, with 
\[\uptheta \simeq \bast\upalpha_{i_{k}}\]
for all $k=1,\dots l\leq \deg \bast f$.  If $l=1$ we will say that $\bast f\in\bast\tv{Z}(\uptheta)$
is a {\bf simple} polynomial diophantine approximation of $\uptheta$.
Write
$\bast\updelta_{i_{k}} =\uptheta - \bast\upalpha_{i_{k}} \in\bast\C_{\upvarepsilon} $, $k=1,\dots l$, and define
\[\upzeta_{d} (\bast f):= \langle \prod_{k=1}^{l}\bast\updelta_{i_{k}} \rangle^{1/d} .\]


\begin{lemm}\label{simpleerror}  Let $\bast f\in\bast\tv{Z}(\uptheta)_{d}$.
Then
\[\upnu_{d} (\bast f) = \mathfrak{m}_{d} (\bast f)\cdot \upzeta_{d} (\bast f). \]
\end{lemm}

\begin{proof} It is enough prove the claim when $\deg (\bast f)=d$.  Assume that $\bast f$ is simple and write $\bast f(X) =\bast a_{d}\prod_{i=1}^{d} (X-\bast\upalpha_{i})$. 
Suppose that $\uptheta\simeq\bast\upalpha =\bast\upalpha_{1}$. Then
\begin{align}\label{nucalcforsimplef}
\upnu_{d} (\bast f) = \langle | \bast a_{d}\prod_{i=1}^{d} (\uptheta -\bast\upalpha_{i})|\rangle^{1/d}= 
\left(|\bast a_{d}|\prod_{i=2}^{d} |\uptheta -\bast\upalpha_{i}|\right)^{1/d}\cdot \upzeta_{d}(\bast f).
\end{align}
The only factors on the right hand side of (\ref{nucalcforsimplef}) which act nontrivially by multiplication on $\upzeta_{d}(\bast f)\in\bstar\PR\R_{\upvarepsilon}$
are $|\bast a_{d}|^{1/d}$ (if the latter is infinite) and those of the form  $|\uptheta-\bast\upalpha_{i}|^{1/d}$ where $\bast \upalpha_{i}\in\bast\C-\bast\C_{\rm fin}$,
in which case 
$ |\uptheta -\bast\upalpha_{i}|^{1/d}\cdot \upzeta_{d}(\bast f) = |\bast\upalpha_{i}|^{1/d}\cdot \upzeta_{d}(\bast f)$.   It follows immediately that 
\[ \left( |\bast a_{d}|\prod_{i=2}^{d} |\uptheta -\bast\upalpha_{i}|\right)^{1/d}\cdot \upzeta_{d}(\bast f)=   
\mathfrak{m}_{d}(\bast f)\cdot \upzeta_{d}(\bast f),\] where we note that 
\begin{enumerate}
\item[-] For $\uptheta\not=0$, $|\bast \upalpha|^{1/d}\in\bast\R_{\rm fin}-\bast\R_{\upvarepsilon}$ and so acts trivially
on $\upzeta_{d}(\bast f)$).  
\item[-] For $\uptheta=0$, $|\bast \upalpha|^{1/d}\in \bast\R_{\upvarepsilon}$ does not occur in $\mathfrak{m}_{d}(\bast f)$.
\end{enumerate}
The argument in the non simple case is identical.
\end{proof}


\begin{lemm}\label{diamonderror} Let $\uptheta, \upeta\in\R$
and let $\bast f\in\bast\tv{Z}(\uptheta )_{d}$,
$\bast g\in\bast\tv{Z}(\upeta )_{e}$.   Then
\begin{align}\label{diamonderrorinequality}
 \upnu_{de} (\bast f\boxtimes\bast g) \leq \left(\mathfrak{m}_{d}(\bast f)^{1-e^{-1}} \mathfrak{m}_{e}(\bast g)\cdot \upnu_{d}(\bast f)^{e^{-1}}\right)+ \left( \mathfrak{m}_{d}(\bast f)\mathfrak{m}_{e}(\bast g)^{1-d^{-1}}\cdot\upnu_{e}(\bast g)^{d^{-1}}\right). 
 \end{align}
\end{lemm}

\begin{proof}  We will show, equivalently, that the unnormalized decay 
\[ \upnu (\bast f\boxtimes\bast g)=\upnu_{de} (\bast f\boxtimes\bast g)^{de}\]
satisfies
\begin{align*}
  \upnu (\bast f\boxtimes\bast g)&\leq \left(\mathfrak{m}(\bast f)^{e-1} \mathfrak{m}(\bast g)^{d}\cdot \upnu(\bast f)\right)+ \left( \mathfrak{m}(\bast f)^{e}\mathfrak{m}(\bast g)^{d-1}\cdot\upnu(\bast g)\right). 
 \end{align*}
 In what follows we will denote $\upzeta(\bast f)=\upzeta_{d}(\bast f)^{d}$, $\upzeta(\bast g)=\upzeta_{e}(\bast g)^{e}$. 
Without loss of generality we may assume that $\deg (\bast f)=d$ and $\deg (\bast g)=e$.  Assume first that $\bast f$, $\bast g$ are simple
and $\not=0$.  Then there exist unique roots, say $\bast\upalpha=\bast\upalpha_{1}$
and $\bast\upbeta=\bast\upbeta_{1}$, so that $\uptheta = \bast\upalpha+\bast\updelta$,  $\upeta = \bast\upbeta+\bast\upvarepsilon$.
By Lemma \ref{simpleerror}, $\upnu (\bast f) = \mathfrak{m}(\bast f)\cdot\upzeta(\bast f)$,
$\upnu (\bast g) = \mathfrak{m}(\bast g)\cdot \upzeta(\bast g)$.
Now
\begin{align}\label{diamonderror1} |(\bast f\boxtimes\bast g)(\uptheta\upeta )| &=a_{d}^{e}b_{e}^{d}
\bigg(\prod_{i=2}^{d}\big(\prod_{j=1}^{e}|\uptheta\upeta -\bast\upalpha_{i}\bast\upbeta_{j}|\big)\bigg)\times 
\prod_{j=1}^{e}(\uptheta\upeta-\bast\upalpha\bast\upbeta_{j}).
\end{align}
The last factor in (\ref{diamonderror1}) can be written 
\[ \prod_{j=1}^{e}|\uptheta\upeta-\bast\upalpha\bast\upbeta_{j}| = |\bast\upalpha\bast\upvarepsilon+\bast\upbeta\bast\updelta+
\bast\updelta\bast\upvarepsilon| \cdot\prod_{j=2}^{e}|\uptheta\upeta-\bast\upalpha\bast\upbeta_{j}|.\] 
Projecting $|\bast\upalpha\bast\upvarepsilon+\bast\upbeta\bast\updelta+
\bast\updelta\bast\upvarepsilon|$  to $\bstar\PR\R_{\upvarepsilon}$ gives 
\[ \langle|\bast\upalpha\bast\upvarepsilon+\bast\upbeta\bast\updelta+
\bast\updelta\bast\upvarepsilon|\rangle\leq \upzeta(\bast f)+ \upzeta(\bast g)\] and therefore
\begin{align*}
\big\langle\prod_{j=1}^{e}|\uptheta\upeta-\bast\upalpha\bast\upbeta_{j}|\big\rangle & \leq
 \bigg( \prod_{j=2}^{e}|\uptheta\upeta-\bast\upalpha\bast\upbeta_{j}|\bigg)\cdot \big( \upzeta(\bast f)+ \upzeta(\bast g)\big) \\
  & = \bigg(|\bast\upalpha|^{e-1}\prod_{j=2}^{e}\big|\big(1+ (\bast\updelta/\bast\upalpha)\big)\upeta-\bast\upbeta_{j}\big|\bigg)\cdot
  \big(\upzeta(\bast f)+ \upzeta(\bast g)\big) \\
  & = \bigg(|\bast\upalpha|^{e-1}\prod_{j=2}^{e}\big|\upeta-\bast\upbeta_{j}\big|\bigg)\cdot
  \big(\upzeta(\bast f)+ \upzeta(\bast g) \big)\\
  & \leq |\bast\upalpha|^{e}\cdot\mathfrak{m}(\bast g)\cdot\big(\upzeta(\bast f)+ \upzeta(\bast g)\big).
\end{align*}
In the above, we were able to insert an extra factor $|\bast\upalpha|$, and replace $|\upeta-\bast\upbeta_{j}|$ by $|\bast\upbeta_{j}|$, since both 
$|\bast\upalpha|$ and $|\upeta|$ belong to $(\bast\R_{\rm fin})^{\times}_{+}$
and hence act trivially by multiplication in $\bstar\PR\R_{\upvarepsilon}$.  The remaining
factors $\prod_{j=1}^{e}|\uptheta\upeta -\bast\upalpha_{i}\bast\upbeta_{j}|$, $i=2,\dots d$, appearing in (\ref{diamonderror1}) can be treated similarly, as follows:  

\vspace{3mm}

\noindent \fbox{\small {\it Case} 1} $\bast\upalpha_{i}$ is bounded, non infinitesimal.  Write $|\uptheta\upeta -\bast\upalpha_{i}\bast\upbeta_{j}|=|\bast\upalpha_{i}||(\uptheta/\bast\upalpha_{i})\upeta-\bast\upbeta_{j}|$.
We may replace $|(\uptheta/\bast\upalpha_{i})\upeta-\bast\upbeta_{j}|$ by $|\bast\upbeta_{j}|$ unless the latter is infinitesimal.  If 
$|\bast\upbeta_{j}|$ is infinitesimal,
$\mathfrak{m}(\bast g)$ does not include the factor $|\bast\upbeta_{j}|$, and moreover, $|(\uptheta/\bast\upalpha)\upeta-\bast\upbeta_{j}|$ is bounded and so may be omitted anyway.  Thus the factor $\prod_{j=1}^{e}|\uptheta\upeta -\bast\upalpha_{i}\bast\upbeta_{j}|$ may be replaced by $|\bast \upalpha_{i}|^{e}\mathfrak{m}(\bast g)\leq \max (|\bast \upalpha_{i}|,1)^{e}\mathfrak{m}(\bast g)$.

\vspace{3mm}

\noindent \fbox{\small {\it Case} 2} $\bast\upalpha_{i}$ is infinitesimal.   There is no need to factor out $|\bast\upalpha_{i}|$ since it will not appear in $\mathfrak{m}(\bast f)$.
The factor $|\uptheta\upeta -\bast\upalpha_{i}\bast\upbeta_{j}|$ will either be 
\begin{itemize}
\item[-] infinite: in which case $\bast\upbeta_{j}$
is infinite, and we can replace $|\uptheta\upeta -\bast\upalpha_{i}\bast\upbeta_{j}|$ by $|\bast\upbeta_{j}|$, with the possible cost of an inequality $\leq$.
\item[-]  bounded: in which case $\bast\upbeta_{j}$ may be either finite or infinite.  We can in either case replace $|\uptheta\upeta -\bast\upalpha_{i}\bast\upbeta_{j}|$ by $|\bast\upbeta_{j}|$, with the possible cost of an $\leq$ inequality in the event that $\bast\upbeta_{j}$ is infinite. 
\end{itemize}
In this case, it follows then that the factor $\prod_{j=1}^{e}|\uptheta\upeta -\bast\upalpha_{i}\bast\upbeta_{j}|$ may be replaced by $\mathfrak{m}(\bast g)= \max (|\bast \upalpha_{i}|,1)^{e}\mathfrak{m}(\bast g)$, with the possible cost of the inequality $\leq$.

\vspace{3mm}

\noindent \fbox{\small {\it Case} 3}  $\bast\upalpha_{i}$ is infinite.  Factor out $|\bast\upalpha_{i}|$ as in {\footnotesize \fbox{{\it Case} 1}}.
Since $(\uptheta/\bast\upalpha)\upeta$ is infinitesimal, we can replace $|(\uptheta/\bast\upalpha)\upeta-\bast\upbeta_{j}|$ by $|\bast\upbeta_{j}|$, except when the latter is infinitesimal, but then $|\bast\upbeta_{j}|$ does not appear in $\mathfrak{m}(\bast g)$ and we can omit it, again with a possible cost
of an $\leq$ inequality.  Once again, the associated factor can be replaced by $\max (|\bast \upalpha_{i}|,1)^{e}\mathfrak{m}(\bast g)$.

\vspace{3mm}

\noindent From the analysis provided above, we conclude
\begin{align*}
\upnu (\bast f\boxtimes\bast g) & \leq\mathfrak{m}(\bast f)^{e}\mathfrak{m}(\bast g)^{d}\cdot 
\big(\upzeta(\bast f)+ \upzeta(\bast g)\big) \\
 & = \left(\mathfrak{m}(\bast f)^{e-1} \mathfrak{m}(\bast g)^{d}\cdot\upnu(\bast f)\right)+ \left( \mathfrak{m}(\bast f)^{e} \mathfrak{m}(\bast g)^{d-1}\cdot \upnu(\bast g)\right)
\end{align*}
as claimed.  Now consider the case of nonsimple polynomial diophantine approximations $\bast f$, $\bast g$
in which $\bast\upalpha_{i_{1}}\simeq \uptheta,\dots ,\bast\upalpha_{i_{r}}\simeq \uptheta$ and 
$\bast \upbeta_{j_{1}}\simeq\upeta ,\dots \bast \upbeta_{j_{s}}\simeq\upeta $.  Arguing as we did in the simple case we obtain the inequality
\[ \upnu (\bast f\boxtimes\bast g)\leq \mathfrak{m}(\bast f)^{e}\cdot  \mathfrak{m}(\bast g)^{d}
\prod_{k=1}^{r}\prod_{l=1}^{s} \langle|\bast\updelta_{i_{k}}+\bast\upvarepsilon_{j_{l}}|\rangle  \]
where $\uptheta = \bast \upalpha_{i_{k}}+\bast \updelta_{i_{k}}$,  $\upeta = \bast \upbeta_{j_{l}}+\bast \upvarepsilon_{j_{l}}$.
But
\begin{align}   \prod_{k=1}^{r}\prod_{l=1}^{s} \langle|\bast\updelta_{i_{k}}+\bast\upvarepsilon_{j_{l}}|\rangle &  \leq
\prod_{k=1}^{r}\prod_{l=1}^{s}\big( \langle|\bast\updelta_{i_{k}}|\rangle+ \langle|\bast\upvarepsilon_{j_{l}}|\rangle\big) \nonumber \\
& = \sum \langle |\bast\updelta_{i_{1}}^{u_{1}}\cdots \bast\updelta_{i_{r}}^{u_{s}}\cdot \bast \upvarepsilon_{j_{1}}^{v_{1}}\cdots\bast \upvarepsilon_{j_{s}}^{v_{s}}|\rangle
\label{sumproductoferr}\end{align}
where the sum is over exponents $(u_{1},\dots u_{s}; v_{1},\dots ,v_{s})$ whose sum is $rs$ and for which $0\leq u_{1},\dots , u_{r}\leq s$
and $0\leq v_{1},\dots , v_{s}\leq r$.  But each summand in (\ref{sumproductoferr}) contains a factor of either $\langle\prod_{k=1}^{r}|\bast\updelta_{i_{k}}|\rangle$
or $\langle\prod_{l=1}^{s}|\bast\upvarepsilon_{j_{l}}|\rangle$, and hence is bounded by one or the other.  It follows that
\[  \prod_{k=1}^{r}\prod_{l=1}^{s} \langle|\bast\updelta_{i_{k}}+\bast\upvarepsilon_{j_{l}}|\rangle\leq 
\langle\prod_{k=1}^{r}|\bast\updelta_{i_{k}}|\rangle+ \langle\prod_{l=1}^{s}|\bast\upvarepsilon_{j_{l}}|\rangle
=\upzeta(\bast f)+ \upzeta(\bast g) ,\] 
from which we derive the desired conclusion in this case.
\end{proof}
\begin{theo}[Ordinary Approximate Ideal Arithmetic--Finite Degree]\label{OrdFinDegIdeoArith} There is a well-defined Cauchy bilinear map given by the $\boxtimes$-product
\[   \bast\tv{Z}^{\upmu}_{\upmu^{e-1}\cdot\upnu^{e}}(\uptheta )_{d} \times \bast\tv{Z}^{\upnu}_{\upmu^{d}\cdot\upnu^{d-1}}(\upeta )_{e}\stackrel{\boxtimes}{\longrightarrow}  \bast\tv{Z}^{\upmu\cdot\upnu}(\uptheta\upeta )_{de}\]
which agrees with the linear approximate ideal product when $d=e=1$.
\end{theo}

\begin{proof}  This is a fairly straightforward consequence of Lemma \ref{diamonderror}.  If $\bast f\in  \bast\tv{Z}^{\upmu}_{\upmu^{e-1}\cdot\upnu^{e}}(\uptheta )_{d}$,
$\bast g\in\bast\tv{Z}^{\upnu}_{\upmu^{d}\cdot\upnu^{d-1}}(\upeta )_{e}$ then 
$\upnu_{de}(\bast f\boxtimes\bast g)\in\bstar\PR\R_{\upvarepsilon}$ since $\upnu_{d}(\bast f)\leq \upmu^{e-1}\cdot\upnu^{e}$ and
$\upnu_{e}(\bast g)\leq \upmu^{d}\cdot\upnu^{d-1}$ imply that the right hand side of (\ref{diamonderrorinequality}) is infinitesimal.
Moreover, by Proposition \ref{Mahlersubresultmult},
\[\upmu\cdot\upnu<\upmu_{de}(\bast f\boxtimes\bast g).\]
\end{proof}

\begin{coro}[Frobenius Approximate Ideal Arithmetic--Finite Degree]\label{FrobFinDegreeCor}
Let $d,e\in\N$ and suppose that $\bar{\upmu}\geq\bar{\upnu}$.  Then
\begin{enumerate}
\item If $d,e\not=1$, there is a well-defined Cauchy bilinear pairing
\[   \bast\tv{Z}^{\bar{\upmu}}_{\bar{\upnu}}(\uptheta )_{d} \times\bast\tv{Z}^{\bar{\upnu}}_{\bar{\upnu}}(\upeta )_{e}\stackrel{\boxtimes}{\longrightarrow}  \bast\tv{Z}^{\bar{\upnu}}(\uptheta\upeta )_{de}.\]  
\item If $d=1$ and $e\not=1$ there is a well-defined Cauchy bilinear pairing
\[   \bast\tv{Z}^{\bar{\upmu}}_{\bar{\upnu}}(\uptheta )_{1} \times\bast\tv{Z}^{\bar{\upnu}}_{\bar{\upmu}}(\upeta )_{e}\stackrel{\boxtimes}{\longrightarrow}  \bast\tv{Z}^{\bar{\upnu}}(\uptheta\upeta )_{e}\quad  \]
and a similar pairing in the case $d\not=1$ and $e=1$.
\item If $d=1=e$, there is a well-defined Cauchy bilinear pairing
\[   \bast\tv{Z}^{\bar{\upmu}}_{\bar{\upnu}}(\uptheta )_{1} \times\bast\tv{Z}^{\bar{\upnu}}_{\bar{\upmu}}(\upeta )_{1}\stackrel{\boxtimes}{\longrightarrow}  \bast\tv{Z}^{\bar{\upnu}}(\uptheta\upeta )_{1} \]
which is just the Frobenius image of the linear (multiplicative) approximate ideal product.
\end{enumerate}
\end{coro}

\begin{proof}  This is immediate from Theorem \ref{OrdFinDegIdeoArith}, where we use Proposition \ref{tropsubistropprod}
to conclude that $\bar{\upmu}\cdot\bar{\upnu}=\bar{\upnu}$.
\end{proof}

\section{Polynomial Approximate Ideal Arithmetic II: Infinite Degree}\label{polygdasec2}

We now consider growth-decay arithmetic for infinite degree polynomial approximations, which takes place within
$\bast\tilde{\tv{Z}}$.   Given $\bast f=\bast\{ f_{i}\}\in\bast\tilde{\tv{Z}}$, by a {\bf root} of $\bast f$
we mean $\bast \upalpha = \bast\{\upalpha_{i}\}\in\bast\bar{\Q}$ such that $\bast \{f_{i}(\upalpha_{i})\}=0$ in 
$\bast\bar{\Q}$.  
The set of roots of $\bast f$ (in which roots are repeated according to multiplicity) can be identified with the ultraproduct 
\[ \mathbbm{r}(\bast f):=\prod_{\mathfrak{u}} \mathbbm{r}(f_{i})\]  where $\mathfrak{u}$
is the ultrafilter used throughout this paper to define the ultrapowers $\bast\Z,\bast\R$, etc.  In particular, if
$\bast f\in\bast\tv{Z}$, $|\mathbbm{r}(\bast f)|<\infty$; otherwise, for $\bast f\in  \bast\tilde{\tv{Z}}-\bast\tv{Z}$,
 $|\mathbbm{r}(\bast f)|>\aleph_{0}$.
The infinite degree case is complicated by the existence of diophantine approximations $\bast f\in\bast\tilde{\tv{Z}}(\uptheta )$
having infinitely many roots for which 
$\uptheta\not\simeq \bast \upalpha$. 
Consider $\uptheta\not=0$ and $f=a_{d}\prod (X-\upalpha_{i})\in\Z [X]$ standard.  Define the  $\boldsymbol\uptheta${\bf -Mahler measure}
and the $\boldsymbol\uptheta${\bf -disk measure} by
\[  \mathfrak{m}(f)_{\uptheta} :=a_{d}\prod \max (|\uptheta -\upalpha_{i}|,1),\quad \mathfrak{z}(f)_{\uptheta} :=
a_{d}\prod_{|\uptheta -\upalpha_{i}|\leq 1} |\uptheta -\upalpha_{i}|=
|f(\uptheta )|/ \mathfrak{m}_{\uptheta}(f)  .\]
Notice that $\mathfrak{m}(f)_{\uptheta}=\mathfrak{m}(f_{\uptheta})$ and $\mathfrak{z}(f)_{\uptheta}=\mathfrak{z}(f_{\uptheta})$,
 where $f_{\uptheta}(X):=f(X-\uptheta)$.
These definitions extend in the usual way to $\bast f\in\bast \tilde{\tv{Z}}$.   If $\bast f\in\bast \tilde{\tv{Z}}(\uptheta )_{\bast d}$,
the analogue of Lemma \ref{simpleerror}
\[  \upnu_{\bast d} (\bast f) = \mathfrak{m}_{\bast d}(\bast f)_{\uptheta}\cdot \upzeta_{\bast d}(\bast f)_{\uptheta}  \]
(where $ \upzeta(\bast f)_{\uptheta}:=\langle \mathfrak{z}(f)_{\uptheta} \rangle$ and the sub index $\bast d$ indicates the degree normalized versions) follows by definition.

Given $f\in\Z[X]$, a $\boldsymbol\uptheta$-{\bf disk root} is a root $\upalpha$ of $f$ for which $|\uptheta -\upalpha|<1$.
We may extend this definition to $\bast f\in\bast\tilde{\tv{Z}}(\uptheta )$, declaring that a root  $\bast\upalpha $ of $\bast f$ 
is a $\uptheta$-disk root if $|\uptheta -\bast \upalpha|<1$. 
Let $\Updelta_{\uptheta}(\bast f)$ denote the set of $\uptheta$-disk roots of $\bast f$; note that we may identify
$\Updelta_{\uptheta}(\bast f)$ as an ultraproduct:
\[  \Updelta_{\uptheta}(\bast f) = \prod_{\mathfrak{u}}\Updelta_{\uptheta}( f_{i}). \]
It follows from the definitions that
\[ \mathfrak{z}(\bast f)_{\uptheta} =\prod_{ \Updelta_{\uptheta}(\bast f)} |\uptheta-\bast\upalpha|  :=
\bast\big\{  \prod_{\upalpha\in \Updelta_{\uptheta}( f_{i})} |\uptheta-\upalpha|  \big\}.
\]


We will call a root $\bast\upalpha$ of $\bast f\in\bast\tilde{\tv{Z}}(\uptheta )$ which satisfies $\uptheta\simeq\bast\upalpha$ a $\boldsymbol\uptheta$-{\bf root},
the set of which is denoted 
\[\mathbb{r}_{\uptheta}(\bast f)\subset\Delta_{\uptheta}(\bast f).\]
Note that $\mathbb{r}_{\uptheta}(\bast f)$ is not in general an ultraproduct: rather, it may be expressed as a union of ultraproducts
\[ \mathbb{r}_{\uptheta}(\bast f) = \bigcup_{\bast \upvarepsilon}\prod_{\mathfrak{u}}  \mathbb{r}_{\upvarepsilon_{i}}( f_{i}) \]
where $\bast \upvarepsilon = \bast \{ \upvarepsilon_{i}\}\in(\bast\R_{\upvarepsilon})_{+}$ and  $\mathbb{r}_{\upvarepsilon_{i}}( f_{i})$ is
the set of roots of $f_{i}$ with absolute value less than $\upvarepsilon_{i}$.

We distinguish two classes of polynomial diophantine approximations of possibly infinite degree.
We say that $\bast f=\bast\{ f_{i}\}\in\bast\tilde{\tv{Z}}(\uptheta )$ is
\begin{itemize}
\item {\bf pure} if  
$\Updelta_{\uptheta}(\bast f)-\mathbb{r}_{\uptheta}(\bast f)$
is finite or equivalently, 
\begin{align}\label{ultraqualofroots} \mathbb{r}_{\uptheta}(\bast f) = \prod_{\mathfrak{u}}  \mathbb{r}_{\upvarepsilon_{i}}( f_{i})
\end{align}
for some $\bast \upvarepsilon\in (\bast\R_{\upvarepsilon})_{+}$.   In this case, the contributions
from $\Updelta_{\uptheta}(\bast f)-\mathbb{r}_{\uptheta}(\bast f)$ to $\upzeta(\bast f)_{\uptheta}$ can be omitted.
The set of pure approximations defines a Cauchy sub semigroup:
\[ \bast\tv{Z}(\uptheta )\subset  \bast\tilde{\tv{Z}}(\uptheta )^{\text{\tt p}}\subset\bast\tilde{\tv{Z}}(\uptheta ).\] 
\item {\bf mixed} if 
$\Updelta_{\uptheta}(\bast f) -\mathbb{r}_{\uptheta}(\bast f)$ is infinite.  Here the contributions to $\upzeta(f)_{\uptheta}$ 
coming from $\Updelta_{\uptheta}(\bast f)-\mathbb{r}_{\uptheta}(\bast f)$ cannot be ignored. The set of mixed approximations defines a Cauchy sub semigroup
\[ \bast\tilde{\tv{Z}}(\uptheta )^{\text{\tt m}} \subset \bast\tilde{\tv{Z}}(\uptheta ).  \] 
\end{itemize}
Notice that
\[ \bast\tilde{\tv{Z}}(\uptheta )= \bast\tilde{\tv{Z}}(\uptheta )^{\text{\tt p}} \sqcup\bast\tilde{\tv{Z}}(\uptheta )^{\text{\tt m}} \]
as Cauchy semigroups and $\bast\tilde{\tv{Z}}(\uptheta )^{\text{\tt p}}\cdot \bast\tilde{\tv{Z}}(\uptheta )^{\text{\tt m}}\subset\bast\tilde{\tv{Z}}(\uptheta )^{\text{\tt m}}$.

 In order to state the main theorems regarding infinite degree approximate ideal arithmetic, we need to introduce
 the following notation: 
\[ \bast\tilde{\tv{Z}}_{\upnu^{\bast e}}^{\upmu }(\uptheta )_{\bast d}:=\{  \bast f\in \bast\tilde{\tv{Z}}(\uptheta )_{\bast d}|\;
\upmu_{\bast d}(\bast f)>\upmu\in \bstar\PR\R_{\upvarepsilon}, \upnu_{\bast d}(\bast f)^{1/\bast e}\leq \upnu
 \}\]
 where $\bast e\in\bast\Z$.  The exponential expression  $\upnu^{\bast e}$ is not necessarily defined since $\bast e$ may be infinite,
 which is why $\bast\tilde{\tv{Z}}_{\upnu^{\bast e}}^{\upmu }(\uptheta )_{\bast d}$ requires definition.

\begin{theo}[Ordinary Polynomial Approximate Ideal Arithmetic--Pure]\label{ordininfdegsemipure}  Let $\bast f\in \bast\tilde{\tv{Z}}(\uptheta )^{\text{\tt p}}_{\bast d}$, $\bast g\in \bast\tilde{\tv{Z}}(\upeta )^{\text{\tt p}}_{\bast e}$.  Then 
\begin{align}\label{semipureineq} \upnu_{\bast d\bast e} (\bast f\boxtimes\bast g) \leq \left(\mathfrak{m}_{\bast d}(\bast f)^{1-\bast e^{-1}} 
\mathfrak{m}_{\bast e}(\bast g)\cdot \upnu_{\bast d}(\bast f)^{\bast e^{-1}}\right)+ \left( \mathfrak{m}_{\bast d}(\bast f) \mathfrak{m}_{\bast e}(\bast g)^{1-\bast d^{-1}}\cdot \upnu_{\bast e}(\bast g)^{\bast d^{-1}}\right). 
\end{align}
In particular, there is a well-defined map
\[   \bast\tilde{\tv{Z}}^{\upmu}_{(\upmu^{1-\bast e^{-1}}\cdot\upnu)^{\bast e}}(\uptheta )^{\text{\tt p}}_{\bast d} \times\bast\tilde{\tv{Z}}^{\upnu}_{(\upmu\cdot\upnu^{1-\bast d^{-1}})^{\bast d}}(\upeta )^{\text{\tt p}}_{\bast e}\stackrel{\boxtimes}{\longrightarrow}  \bast\tilde{\tv{Z}}^{\upmu\cdot\upnu}(\uptheta\upeta )_{\bast d\bast e}.\]
\end{theo}

\begin{proof}  The proof involves the adaptation of the proof of Lemma \ref{diamonderror} to the infinite degree case.
Let $\bast f=\bast \{ f_{i}\}\in \bast\tilde{\tv{Z}}(\uptheta )^{\text{\tt p}}_{\bast d}$, $\bast g=\bast \{ g_{i}\}\in \bast\tilde{\tv{Z}}(\upeta )^{\text{\tt p}}_{\bast e}$.
First, by purity, only $\uptheta$-roots, $\upeta$-roots are required to define 
$\upzeta_{\bast d}(\bast f)_{\uptheta}$, $\upzeta_{\bast e}(\bast g)_{\upeta}$: that is,
\[ \upzeta_{\bast d}(\bast f)_{\uptheta} = 
\left\langle\bast\big\{ \prod_{\upalpha_{j_{i}}\in \mathbb{r}_{x_{i}}( f_{i})}|\uptheta -\upalpha_{j_{i}}|\big\}\right\rangle = 
\left\langle\bast\big\{ \prod_{\upalpha_{j_{i}}\in \mathbb{r}_{x_{i}}( f_{i})}|\updelta_{j_{i}}|\big\}\right\rangle
\]
where $\bast x =\bast \{x_{i}\}\in\bast\R_{\upvarepsilon}$ and we have used the fact that $\mathbb{r}_{\uptheta}(\bast f)$
is an ultraproduct, see (\ref{ultraqualofroots})).  Similarly,
\[ \upzeta_{\bast e}(\bast g)_{\upeta} = 
\left\langle\bast\big\{ \prod_{\upbeta_{k_{i}}\in \mathbb{r}_{y_{i}}( g_{i})}|\upeta -\upbeta_{k_{i}}|\big\}\right\rangle = 
\left\langle\bast\big\{ \prod_{\upbeta_{k_{i}}\in \mathbb{r}_{y_{i}}( g_{i})}|\upvarepsilon_{k_{i}}|\big\}\right\rangle .
\]
Now $|(\bast f\boxtimes \bast g)(\uptheta\upeta )|^{1/\bast d\bast e}$ is the class of the sequence
\begin{align}\label{fdiamondgcalc} 
\left\{ |a_{d_{i}}|^{1/d_{i}}|b_{e_{i}}|^{1/e_{i}}
\bigg( \prod_{j_{i}=1}^{d_{i}}\prod_{k_{i}=1}^{e_{i}} |\uptheta\upeta - \upalpha_{j_{i}}\upbeta_{k_{i}}| \bigg)^{1/d_{i}e_{i}}\right\} =& \nonumber \\
\left\{ P_{i}^{1/d_{i}e_{i}}\times |a_{d_{i}}|^{1/d_{i}}|b_{e_{i}}|^{1/e_{i}}
  \bigg(\prod_{\upalpha_{j_{i}}\in \mathbb{r}_{x_{i}}( f_{i})}\prod_{\upbeta_{k_{i}}\in \mathbb{r}_{y_{i}}( g_{i})}| \upalpha_{j_{i}}\upvarepsilon_{k_{i}} +\upbeta_{k_{i}}\updelta_{j_{i}} + \updelta_{j_{i}}\upvarepsilon_{k_{i}}|\bigg)^{1/d_{i}e_{i}} \right\}
\end{align}
where $P_{i}$ is the product of the $|\uptheta\upeta - \upalpha_{j_{i}}\upbeta_{k_{i}}|$ where either $\upalpha_{k_{i}}\not\in \mathbb{r}_{x_{i}}(f_{i})$
or $\upbeta_{k_{i}}\not\in \mathbb{r}_{y_{i}}(g_{i})$.  Upon passage to $\bstar\PR\R_{\upvarepsilon}$, the expression in the parenthesis in  (\ref{fdiamondgcalc}) is bounded by either 
$\langle \bast \{ d_{i}e_{i}\prod |\updelta_{j_{i}}| \}\rangle$ or $\langle\bast \{ d_{i}e_{i}\prod |\upvarepsilon_{k_{i}} |\}\rangle$.  In particular, we obtain
\[ \upnu_{\bast d\bast e} (\bast f\boxtimes\bast g) \leq \bast P^{1/\bast d\bast e} |\bast a|^{1/\bast d}|\bast b|^{1/\bast e}\cdot 
\left(\upzeta_{\bast d}(\bast f)_{\uptheta}^{1/\bast e}+ \upzeta_{\bast e}(\bast g)_{\upeta}^{1/\bast d}\right) .\]
We are left with analyzing the remaining factor  $\bast P^{1/\bast d\bast e} |\bast a|^{1/\bast d}|\bast b|^{1/\bast e}$, and to this end, we return to a representative sequence
\begin{align}\label{whatsleftaftergothz}
|a_{d_{i}}|^{1/d_{i}}|b_{e_{i}}|^{1/e_{i}}
\bigg( \prod |\uptheta\upeta - \upalpha_{j_{i}}\upbeta_{k_{i}}| \bigg)^{1/d_{i}e_{i}}.
\end{align}
 The first remark is that we can insert a factor
of the form
\begin{align*}
  \bigg( \prod_{\upalpha_{j_{i}}\in \mathbb{r}_{x_{i}}( f_{i})}\prod_{\upbeta_{k_{i}}\in \mathbb{r}_{y_{i}}( g_{i})} |\upalpha_{j_{i}}\upbeta_{k_{i}}| \bigg)^{1/d_{i}e_{i}} & \leq  |\uptheta\upeta| \bigg( \prod_{\upalpha_{j_{i}}\in \mathbb{r}_{x_{i}}( f_{i})}\prod_{\upbeta_{k_{i}}\in \mathbb{r}_{y_{i}}( g_{i})}
   \big|1+|\updelta_{j_{i}}\uptheta^{-1}|\big|\big|1+|\upvarepsilon_{k_{i}}\upeta^{-1}|\big| \bigg)^{1/d_{i}e_{i}} \\
   &  \longrightarrow 4|\uptheta\upeta| .
  \end{align*}
  We can do this, since we will be multiplying the result with $(\upzeta_{\bast d}(\bast f)_{\uptheta}^{1/\bast e}+ \upzeta_{\bast e}(\bast g)_{\upeta}^{1/\bast d})\in\bstar\PR\R_{\upvarepsilon}$, and thus we are free to modify (\ref{whatsleftaftergothz}) 
by any uniformly bounded factor without changing the bound we have so far on $ \upnu_{\bast d\bast e} (\bast f\boxtimes\bast g)$.
Now we claim that for each factor $ |\uptheta\upeta - \upalpha_{j_{i}}\upbeta_{k_{i}}|$ in (\ref{whatsleftaftergothz}),
\begin{align}\label{theCbound}
|\uptheta\upeta-\upalpha_{j_{i}}\upbeta_{k_{i}}|=|\upalpha_{j_{i}}\upbeta_{k_{i}}||(\uptheta\upeta)/(\upalpha_{j_{i}}\upbeta_{k_{i}}) -  1| \leq C\max (1,|\upalpha_{j_{i}}|)\cdot \max (1,|\upbeta_{k_{i}}|)
\end{align}
for some constant $C>0$ independent of $i$ and depending only on $\uptheta,\upeta$.  If both $|\upalpha_{j_{i}}|,|\upbeta_{k_{i}}|\geq 1$, this is clear
since $|(\uptheta\upeta)/(\upalpha_{j_{i}}\upbeta_{k_{i}}) -  1|$ has a uniform upper bound for such roots.  This is true more generally whenever
$|\upalpha_{j_{i}}\upbeta_{k_{i}}|\geq 1$.  If $|\upalpha_{j_{i}}\upbeta_{k_{i}}|< 1$, then $ |\uptheta\upeta-\upalpha_{j_{i}}\upbeta_{k_{i}}|$
is uniformly bounded by some constant $C$, which {\it a fortiori} satisfies (\ref{theCbound}).  Therefore we may replace each factor of
$|\uptheta\upeta-\upalpha_{j_{i}}\upbeta_{k_{i}}|$ occurring in $P_{i}$ with $C\max (1,|\upalpha_{j_{i}}|)\cdot \max (1,|\upbeta_{k_{i}}|)$, at the cost
of an inequality.  It follows that
\begin{align*} \upnu_{\bast d\bast e} (\bast f\boxtimes\bast g)\leq &  \mathfrak{m}_{\bast d}(\bast f)
\mathfrak{m}_{\bast e}(\bast g)\cdot (\upzeta_{\bast d}(\bast f)_{\uptheta}^{\bast e^{-1}}+ \upzeta_{\bast e}(\bast g)_{\upeta}^{\bast d^{-1}})\\
= &  \mathfrak{m}_{\bast d}(\bast f)
\mathfrak{m}_{\bast e}(\bast g)\mathfrak{m}_{\bast d}(\bast f)_{\uptheta}^{-\bast e^{-1}}\cdot \upnu_{\bast d} (\bast f)^{\bast e^{-1}} +\\
& \mathfrak{m}_{\bast d}(\bast f)\mathfrak{m}_{\bast e}(\bast g)\mathfrak{m}_{\bast e}(\bast g)_{\upeta}^{-\bast d^{-1}}\cdot \upnu_{\bast e} (\bast g)^{\bast d^{-1}}
\end{align*}
What remains is to show is that we may replace the $\uptheta$ and $\upeta$ (normalized) Mahler measures by the usual ones.  This follows
from an argument similar to that used to produce the inequality (\ref{theCbound}).
\end{proof}

We now extend the approximate ideal product to include the mixed classes.
For $\uprho\in (0,1]$ let 
\[ \Updelta^{\uprho}_{\uptheta}(\bast f) = \{ \bast \upalpha\in \Updelta_{\uptheta}(\bast f)|\;|\uptheta -\bast \upalpha|<\uprho  \}
\subset \Updelta_{\uptheta}(\bast f).\]
Using $\Updelta^{\uprho}_{\uptheta}(\bast f)$ we define $\mathfrak{z}^{\uprho}_{\bast d}(f)_{\uptheta}$ and 
$\upzeta^{\uprho}_{\bast d}(\bast f)_{\uptheta}$ in the obvious way.

\begin{lemm}\label{Wirsing}  For all $\uprho\in (0,1]$, $\upzeta^{\uprho}_{\bast d}(\bast f)_{\uptheta}=\upzeta_{\bast d}(\bast f)_{\uptheta}$.
\end{lemm}

\begin{proof}  By Wirsing's estimate \cite{Wi} -- using the formulation given in Corollary A.1 in Appendix A.1 of \cite{Bu} -- we have:
\[2^{-(\bast d+1)/\bast d}\cdot (\bast d+1)^{-1/(2\bast d)}\cdot   \max (1,|\uptheta|)^{-1} \cdot \left(\frac{|\bast f(\uptheta)|}{\mathfrak{h}(\bast f)}\right)^{1/\bast d} \leq \mathfrak{z}^{\uprho}_{\bast d}(\bast f)_{\uptheta}\]
and 
\[ \mathfrak{z}^{\uprho}_{\bast d}(\bast f)_{\uptheta}\leq 2^{(\bast d+1)/\bast d}\cdot   \binom{\bast d}{[\bast d/2]}^{1/\bast d}\cdot \uprho^{-1}\max (1,|\uptheta|) \cdot \left(\frac{|\bast f(\uptheta)|}{\mathfrak{h}(\bast f)}\right)^{1/\bast d}. \]
Taking $\langle\cdot\rangle$ of either bound gives immediately
\[  \upzeta^{\uprho}_{\bast d}(\bast f)_{\uptheta}=\upzeta_{\bast d}(\bast f)_{\uptheta} =\upnu_{\bast d} (\bast f)\cdot\upmu_{\bast d} (\bast f). \]
\end{proof}

\begin{theo}[Ordinary Polynomial Approximate Ideal Arithmetic--General]\label{semivirtualarith}   
Let 
$\bast f\in \bast\tilde{\tv{Z}}(\uptheta )_{\bast d}, \bast g\in \bast\tilde{\tv{Z}}(\upeta )_{\bast e}$.  Then 
\[ \upnu_{\bast d\bast e} (\bast f\boxtimes\bast g) \leq \left(\mathfrak{m}_{\bast d}(\bast f)^{1-\bast e^{-1}} \mathfrak{m}_{\bast e}(\bast g)\cdot \upnu_{\bast d}(\bast f)^{\bast e^{-1}}\right)+ \left( \mathfrak{m}_{\bast d}(\bast f)\mathfrak{m}(\bast g)^{1-\bast d^{-1}}\cdot\upnu_{\bast e}(\bast g)^{\bast d^{-1}}\right). \]
In particular, there is a well-defined map
\[   \bast\tilde{\tv{Z}}^{\upmu}_{(\upmu^{1-\bast e^{-1}}\cdot\upnu)^{\bast e}}(\uptheta )_{\bast d} \times\bast\tilde{\tv{Z}}^{\upnu}_{(\upmu\cdot\upnu^{1-\bast d^{-1}})^{\bast d}}(\upeta )_{\bast e}\stackrel{\boxtimes}{\longrightarrow}  \bast\tilde{\tv{Z}}^{\upmu\cdot\upnu}(\uptheta\upeta )_{\bast d\bast e}.\]
\end{theo}

\begin{proof} Choose, $\uprho,\upsigma\in (0,1)$ such that
\[ \Updelta_{\uptheta}^{\uprho}(\bast f)\cdot \Updelta^{\upsigma}_{\upeta}(\bast g)\subset \Updelta_{\uptheta\upeta}(\bast f\boxtimes \bast g ).\]
Now $\upnu_{\bast d\bast e}(\bast f\boxtimes \bast g) = 
\mathfrak{m}_{\bast d\bast e}(\bast f\boxtimes \bast g)_{\uptheta\upeta}\cdot \upzeta_{\bast d\bast e}(\bast f\boxtimes \bast g)_{\uptheta\upeta}$.
By an argument similar to that found in the proof of Theorem \ref{ordininfdegsemipure} by Lemma \ref{Wirsing}, we may bound 
\[  \upzeta_{\bast d\bast e}(\bast f\boxtimes \bast g)_{\uptheta\upeta} \leq   \upzeta^{\uprho}_{\bast d}(\bast f)_{\uptheta} +
 \upzeta^{\upsigma}_{\bast e}( \bast g)_{\upeta} =  \upzeta_{\bast d}(\bast f)_{\uptheta} +
 \upzeta_{\bast e}( \bast g)_{\upeta} :\]
since $ \upzeta_{\bast d\bast e}(\bast f\boxtimes \bast g)_{\uptheta\upeta}$ is calculated using elements of 
$\Updelta_{\uptheta\upeta}(\bast f\boxtimes \bast g )\supset \Updelta^{\uprho}_{\uptheta}(\bast f)\cdot \Updelta^{\upsigma}_{\upeta}(\bast g) $, and being
a product of quantities $\leq 1$, omitting the elements of $\Updelta_{\uptheta\upeta}(\bast f\boxtimes \bast g )$ not belonging to $\Updelta_{\uptheta}(\bast f)\cdot \Updelta_{\upeta}(\bast g) $ makes the result larger. 
The rest of the proof follows that of Theorem  \ref{ordininfdegsemipure}.
\end{proof}

Passing to Frobenius classes gives:

\begin{theo}[Frobenius Polynomial Approximate Ideal Arithmetic]\label{FrobInfDegIdeoArith}Let $\bast d,\bast e\in\bast\N$ and $\bar{\upmu}\geq\bar{\upnu}$. There is a well-defined map
\[   \bast\tilde{\tv{Z}}^{\bar{\upmu}}_{\bar{\upnu}^{\bast e}}(\uptheta )_{\bast d} \times
\bast\tilde{\tv{Z}}^{\bar{\upnu}}_{\bar{\upnu}^{\bast d}}(\upeta )_{\bast e}\stackrel{\boxtimes}{\longrightarrow}  \bast\tilde{\tv{Z}}^{\bar{\upnu}}(\uptheta\upeta )_{\bast d\bast e}.\]
If either $\bast d$ or $\bast e$ is 1, we may adjust the indices in the manner of item (2) of Corollary \ref{FrobFinDegreeCor}.
\end{theo}

Let $\upmu\geq \upnu$ resp. $\bar{\upmu}\geq \bar{\upnu}$ and define the relations
\[\uptheta\, {}^{\bast d\!\!}_{\;\;\upmu}\!\boxtimes^{\!\!\bast e}_{\,\upnu}\,\upeta ,\quad  \uptheta\, {}^{\bast d\!\!}_{\;\;\bar{\upmu}}\!\boxtimes^{\!\!\bast e}_{\,\bar{\upnu}}\,\upeta\]
whenever a product of the type defined in Theorem \ref{semivirtualarith} resp.\ Theorem \ref{FrobInfDegIdeoArith} is nontrivial.
Note that it is not necessarily the case that $\uptheta\, {}^{\bast d\!\!}_{\;\;\upmu}\!\boxtimes^{\!\!\bast e}_{\,\upnu}\,\upeta$ 
$\Rightarrow$  
$\uptheta\, {}^{\bast d\!\!}_{\;\;\bar{\upmu}}\!\boxtimes^{\!\!\bast e}_{\,\bar{\upnu}}\,\upeta$ e.g.\ recall Note \ref{ordinarytofrobenspectra}.
By Corollary \ref{ratspeccoro} we have:

\begin{coro}  Suppose that $\upgamma \gtrdot\uptheta$ and $\upxi \gtrdot\upeta$ by $R,S\in\text{\sf Rat}_{\Z}$, respectively.  Then 
\[  \uptheta\, {}^{\bast d\!\!}_{\;\;\bar{\upmu}}\!\boxtimes^{\!\!\bast e}_{\,\bar{\upnu}}\,\upeta \quad\Longrightarrow \quad
\upgamma\, {}^{\bast d\cdot\deg (R)\!\!}_{\;\;\bar{\upmu}}\!\boxtimes^{\!\!\bast e\cdot\deg (R)}_{\,\bar{\upnu}}\,\upxi .
\]
\end{coro}



There are nontrivial polynomial approximate ideal products in degree 1, coming from the approximate ideal product
of (\ref{gdprodintro}).  
In what follows, we will show
that there exist nontrivial products for (infinite degree) polynomial classes.  The following Proposition shows that $A$-numbers are
always composable with other reals, provided the relevant indices are $A$-exact (the latter term was defined at the beginning of \S \ref{polygdasec}).

\begin{theo}\label{AArithPos}  Let $\upalpha$ be an $A$-number and $\uptheta\in\R$.
\begin{itemize}
\item Suppose that $\bast \breve{d}>\bast\breve{e}$ and $ \bast\tilde{\tv{Z}}^{\bar{\upmu}}_{\bar{\upnu}^{\bast e}}(\uptheta )_{\bast d}$
is nonempty.  Then 
\[ \uptheta\, {}^{\bast d\!\!}_{\;\;\bar{\upmu}}\!\boxtimes^{\!\!\bast e}_{\,\bar{\upnu}}\,\upalpha.\]
\item  Suppose the triple $(\bar{\upmu},\bar{\upnu}^{\bast e},\bast d)$ is $A$-exact and $\bast\tilde{\tv{Z}}^{\bar{\upnu}}_{\bar{\upnu}^{\bast d}}(\uptheta )_{\bast e}$
is not empty.  Then 
\[ \upalpha\, {}^{\bast d\!\!}_{\;\;\bar{\upmu}}\!\boxtimes^{\!\!\bast e}_{\,\bar{\upnu}}\,\uptheta.\]
\end{itemize}
\end{theo}

\begin{proof}  If $\bast \breve{d}>\bast\breve{e}$ then $\bar{\upnu} >(\bar{\upnu}^{\bast d})^{1/\bast e}$, so the triple 
$(\bar{\upmu},\bar{\upnu}^{\bast d},\bast e)$
is $A$-exact and $ \uptheta\, {}^{\bast d\!\!}_{\;\;\bar{\upmu}}\!\boxtimes^{\!\!\bast e}_{\,\bar{\upnu}}\,\upalpha$.
The second bulleted item is trivial.
\end{proof}
The next result is negative: it asserts that $A$-numbers on the $A$-approximative indices, as well as $S$-numbers, may not be composed with one another.



\begin{theo}\label{Anumarith}  Let $\uptheta$ be either an $A$-number or an $S$-number and let $\upeta\in\R$.   Suppose that $\bar{\upmu}\in\bstar\overline{\PR\R}_{\upvarepsilon}$
is not an $\bast e$th root of unity and that if $\uptheta$ is an $A$-number then the triple $(\bar{\upmu},\bar{\upnu}^{\bast e},\bast d)$ is $A$-approximative.
If $ \uptheta\, {}^{\bast d\!\!}_{\;\;\bar{\upmu}}\!\boxtimes{}^{\!\!\bast e}_{\,\bar{\upnu}}\,\upeta$ then 
\[ \bar{\upmu}^{1/\bast e}<\bar{\upnu}^{1/\bast d}.\]  In particular, 
$ \uptheta\, {}^{\bast d\!\!}_{\;\;\bar{\upmu}}\!\boxtimes^{\!\!\bast e}_{\,\bar{\upnu}}\,\upeta$ does {\rm {\bf not}} hold when
\begin{enumerate}
\item[1.] $\bast d=d<\infty$ or $\bast \breve{d}<\bast \breve{e}$.  
\item[2.] $\upeta$ is either an $A$-number for which the triple $(\bar{\upnu},\bar{\upnu}^{\bast d},\bast e)$ is $A$-approximative, or an $S$-number.
\end{enumerate}
\end{theo}

\begin{proof}  Let $\bast f\in  \bast\tilde{\tv{Z}}^{\bar{\upmu}}_{\bar{\upnu}^{\bast e}}(\uptheta )_{\bast d}$;
if $\uptheta$ is an $A$-number then we assume further that $\bast f$ does not belong to the semigroup $\llparenthesis\uptheta\rrparenthesis$ generated by the minimal polynomial of $\uptheta$ (see 
the paragraph preceding Proposition \ref{simplegda}). 
By the classification of the Frobenius nonvanishing spectra, 
\[ \bar{\upmu}_{\bast d} (\bast f)\leq \bar{\upnu}_{\bast d}(\bast f)^{1/\bast d},\]
where the inequality is strict if $\uptheta$ is an $A$-number. 
Since $\bar{\upmu}\in\bstar\overline{\PR\R}_{\upvarepsilon}$
is not an $\bast e$th root of 1, $ \bar{\upmu}^{1/\bast e}< \bar{\upmu}_{\bast d}(\bast f)^{1/\bast e}$.  Moreover, by the hypotheses on $\bast f$,
\begin{align}\label{Anumberineq}
  \bar{\upmu}^{1/\bast e}< \bar{\upmu}_{\bast d}(\bast f)^{1/\bast e}\leq \bar{\upnu}_{\bast d}(\bast f)^{1/\bast d\bast e}\leq \bar{\upnu}^{1/\bast d} .
\end{align}
(The second inequality is not strict to allow for the possibility that $\bar{\upmu}_{\bast d}(\bast f)$ is a $\bast e$th root of unity.)
As a consequence, 
\begin{align}\label{Anumberineq2}\bar{\upmu}< \bar{\upmu}^{1/\bast e} <\bar{\upnu}^{1/\bast d} 
\leq \bar{\upmu}^{1/\bast d}.
\end{align}
If $\bast d=d\in\N$ then $\bar{\upmu}^{1/d}=\bar{\upmu}$ and (\ref{Anumberineq2}) yields $\bar{\upmu}<\bar{\upmu}$, contradiction. 
If $\bast \breve{d}<\bast \breve{e}$, then $\bar{\upmu}^{\bast d/\bast e}$ is defined since $\bast d/\bast e$ is infinitesimal and
\[ \bar{\upmu}<\bar{\upmu}^{\bast d/\bast e}\leq\bar{\upmu}^{1/\bast e}<1.\]   But (\ref{Anumberineq2}) implies
that $\bar{\upmu}^{\bast d/\bast e}<\bar{\upmu}$, contradiction.
Finally, suppose that $\upeta$ is either an $A$-number for which the triple $(\bar{\upnu},\bar{\upnu}^{\bast d},\bast e)$ is not $A$-exact, or an $S$-number.  By (\ref{Anumberineq}) we have that 
\begin{align}\label{2ndAnumberineq} \bar{\upmu}^{1/\bast e}< \bar{\upnu}^{1/\bast d}.
\end{align}
The analogue of (\ref{Anumberineq}) for the triple $(\bar{\upnu},\bar{\upnu}^{\bast d},\bast e)$ gives $\bar{\upnu}^{1/\bast d}\leq\bar{\upnu}^{1/\bast e}\leq
\bar{\upmu}^{1/\bast e}$, implying with (\ref{2ndAnumberineq}) that $\bar{\upmu}^{1/\bast e}<\bar{\upmu}^{1/\bast e}$, again, contradiction.
\end{proof}


 In what follows consider the special case $\bar{\upmu}=\bar{\upnu}$.

\begin{lemm}\label{Sflatlemma}  Let $\upeta$ be trascendental.  Then for every infinite best degree $\bast e\in\bast \N-\N$, there exists $\bar{\upnu}\in\bstar\overline{\PR\R}_{\upvarepsilon}$
such that for all $\bast d$ with $\bast \breve{d}<\bast \breve{e}$,
\[\bast\tilde{\tv{Z}}^{\bar{\upnu}}_{\bar{\upnu}^{\bast d}}(\upeta )_{\bast e}\not=\{ 0\}.\]  
If $\uptheta$ is either a $T$-number or a $U$-number then we may take $\bast \breve{d}=\bast \breve{e}$ as well.  If $\uptheta$ is a $U$-number
then we may take $\bast \breve{d}=\bast \breve{e} =\breve{1}$.
\end{lemm}

\begin{proof}  Let $\bast  \widehat{g}$ be a best class for $\uptheta$ of infinite degree $\bast e$.  Then 
$\bar{\upmu}_{\bast e}(\bast   \widehat{g})\geq\bar{\upnu}_{\bast e}(\bast   \widehat{g})^{1/\bast e}$ and since
$\bast e $ is infinite, $\bar{\upnu}_{\bast e}(\bast   \widehat{g})^{1/\bast d}<\bar{\upmu}_{\bast e}(\bast   \widehat{g})$ for all $\bast d$ with $\bast \breve{d}<\bast \breve{e}$.  
Then the choice \[ \bar{\upnu}\in [\bar{\upnu}_{\bast e}(\bast   \widehat{g})^{1/\bast d},\bar{\upmu}_{\bast e}(\bast   \widehat{g}))\]
will work.  
If $\uptheta$ is either a $T$-number or a $U$-number and $\bast \breve{d}=\bast \breve{e}$ then $\bar{\upmu}_{\bast d}(\bast   \widehat{g})>\bar{\upnu}_{\bast d}(\bast   \widehat{g})^{1/\bast d}$
and we may select $\bar{\upnu}$ in the nontrivial interval $[\bar{\upnu}_{\bast d}(\bast   \widehat{g})^{1/\bast d},\bar{\upmu}_{\bast d}(\bast   \widehat{g}))$.  If $\uptheta$ is a $U$-number then we use a degree $d$ best class, $d\geq \mathfrak{t}$ = the type of $\upeta$, and we may select $\bar{\upnu}$ in the nontrivial interval 
$[\bar{\upnu}_{\breve{1}}(\bast   \widehat{g}),\bar{\upmu}_{\breve{1}}(\bast   \widehat{g}))$.
\end{proof}

\begin{theo}\label{Uproducts}  Let $\uptheta$ be 
a $U$-number and let $\upeta$ be either a $T$-number or a $U$-number.  Then there exists $\bast d\in\bast \N-\N$
and $\bar{\upmu}\in\bstar\overline{\PR\R}_{\upvarepsilon}$ such that $ \uptheta\, {}^{\bast d\!\!}_{\;\;\bar{\upmu}}\!\boxtimes^{\!\!\bast d}_{\,\bar{\upmu}}\,\upeta$.  If both $\uptheta$ and $\upeta$ are $U$-numbers than we may choose $\bast d=d$ finite.
\end{theo}

\begin{proof}  
Choose $\bast d=\bast e$ and $ \bast  \widehat{g}$ with regard to $\upeta$ as in Lemma \ref{Sflatlemma}.
Since $\uptheta$ is a $U$-number, for $d$ sufficiently large, the exponent $\mathfrak{e}(d,H)$ tends to $\infty$
as $H\rightarrow\infty$.  Now let $\bast H:=\mathfrak{h}(\bast  \widehat{g})$; thus
 $[\bast H^{-1/\bast d}]=\bar{\upmu}_{\bast d}(\bast \widehat{g})$.  Let  $\bast f$
minimize $|\bast f(\uptheta )|$ amongst $\bast f \in\bast\tilde{\tv{Z}}_{\bast d}$ having $\mathfrak{h}(\bast f)\leq \bast H$.
Let $\bast\mathfrak{e}:=\mathfrak{e}(\bast d,\bast H)$; then
\[ \bar{\upmu}_{\bast d}(\bast f)\geq \bar{\upmu}_{\bast d}(\bast  \widehat{g}) = \bar{\upnu}_{\bast d}(\bast f)^{1/(\bast d\cdot\bast\mathfrak{e} )} \]
(the last equality follows from the defining equation $|\bast f(\uptheta )|^{1/\bast d}=\bast H^{-\bast\mathfrak{e}}$ of $\bast\mathfrak{e}$).
If $\bast \mathfrak{e}$ is infinite, we are done, for then
\[ \bar{\upnu}(\bast f)^{1/\bast d}<\bar{\upmu}_{\bast d}(\bast  \widehat{g}) \]
so taking 
\[ \bar{\upmu}\in [\bar{\upnu}(\bast f)^{1/\bast d}+\bar{\upnu}(\bast g)^{1/\bast d},\bar{\upmu}_{\bast d}(\bast  \widehat{g}))\] will 
give the relation $ \uptheta\, {}^{\bast d\!\!}_{\;\;\bar{\upmu}}\!\boxtimes^{\!\!\bast d}_{\,\bar{\upmu}}\,\upeta$.  If not, we choose
another best class $\bast  \widehat{g}$ having the same degree for which the associated height sequence $\bast H$ produces
an infinite $\bast\mathfrak{e}$: since $\uptheta$ is a $U$-number, such a sequence can always be produced.  
The argument is identical if $\upeta$ is also a $U$-number and the degree is chosen finite and large enough.
\end{proof}

\begin{coro}  There exist nontrivial ordinary and Frobenius polynomial approximate ideal products, of both infinite degree and finite degree $>1$.  \end{coro}

\begin{proof}  By Thereom \ref{Uproducts} there exist non trival Frobenius products of both finite and infinite degree.  But 
$\uptheta\, {}^{\bast d\!\!}_{\;\;\bar{\upmu}}\!\boxtimes^{\!\!\bast d}_{\,\bar{\upmu}}\,\upeta$ implies $ \uptheta\, {}^{\bast d\!\!}_{\;\;\upmu}\!\boxtimes^{\!\!\bast d}_{\,\upmu}\,\upeta$ for all $\upmu\in\bar{\upmu}$.
\end{proof}

\end{document}